\documentclass[12pt]{amsart}
\usepackage{amsmath}
\usepackage{amssymb}
\usepackage{amsthm}
\usepackage{bm}
\usepackage[english]{babel}
\usepackage{xcolor}

\def\leq {\leqslant}
\def\le {\leqslant}
\def\ge {\geqslant}
\def\geq {\geqslant}

\def\eps{\varepsilon}
\def\kappa{\varkappa}
\def\phi{\varphi}

\newtheorem{theorem}{Theorem}
\newtheorem{corollary}{Corollary}
\newtheorem{prop}{Proposition}
\newtheorem*{theoremA}{Theorem A}
\newtheorem{lemma}{Lemma}

\theoremstyle{definition}
\newtheorem{definition}{Definition}
\newtheorem{remark}{Remark}
\newtheorem{example}{Example}

\newcommand{\mcw}{\mathcal{W}}
\title[On $L_p$-integrability]
{On $L_p$-integrability of functions with general monotone Fourier coefficients}
\thanks{This research was  supported
by the Committee of Science of the Ministry of Science
and Higher Education of the Republic of Kazakhstan 
(Grant No. AP22688236)}

\author{A. Mukanov}
\address{A. Mukanov,
Institute of Mathematics and Mathematical Modeling, 
Pushkin street, 125, Almaty, Kazakhstan}
\email{mukanov.askhat@gmail.com}
\author{E. Nursultanov}
\address{E. Nursultanov,
Kazakhstan Branch of Lomonosov Moscow State University, 
Kazhymukan street, 11, Astana, Kazakhstan}
\email{er-nurs@yandex.ru}

\begin{document}
\begin{abstract}
We introduce new classes of general monotone sequences 
and study their properties. For functions whose 
Fourier coefficients belong to these classes, we establish Hardy–Littlewood-type theorems.
\end{abstract}
\maketitle
\section{Introduction}

\subsection{Hardy-Littlewood theorem and its analogues}

In this paper, 
we explore the relationship between 
an integrable function $f(x)$ defined on the interval $[-\pi,\pi]$ and its Fourier coefficients sequence.  
One  such result is the well-known Hardy-Littlewood theorem \cite{HaLi}, \cite[Chapter 12]{Zy}.
\begin{theoremA}
Let $\{a_k\}_{k=0}^{\infty}$, $\{b_k\}_{k=1}^{\infty}$ be nonincreasing, nonnegative sequences, and
$f(x) \sim \frac{a_0}{2} + \sum\limits_{k=1}^{\infty} (a_k \cos  kx + b_k \sin kx)$.
Then, for any $1 < p < \infty$,
\begin{equation}\label{HL1}
\|f\|_{L_p([-\pi,\pi])} \asymp \frac{a_0}{2} + \left(\sum_{k=1}^{\infty} k^{p-2}(a_k^p + b_k^p)  \right)^{\frac{1}{p}}.
\end{equation} 
\end{theoremA}
The corresponding result for the function $f(x) \sim \sum\limits_{k=-\infty}^{\infty} c_k e^{i kx}$, where
$\{c_k\}_{k=0}^{\infty}$, $\{c_{-k}\}_{k=0}^{\infty}$ are nonincreasing, nonnegative sequences, can be stated as follows:
\begin{equation}\label{HL2}
\|f\|_{L_p([-\pi, \pi])} \asymp \left(\sum_{k = -\infty}^{\infty} (|k|+1)^{p-2}|c_k|^p\right)^{\frac{1}{p}}, \quad 1 < p < \infty.
\end{equation}

Equivalences \eqref{HL1}, \eqref{HL2} were extended in many papers, e.g., see 
\cite{BeDyTi, DyMuTi, DyNuKa, DyTiJMAA, DyTiWaterman, 
DyTi, DyTi22, GrSaSa, Nu, NuTi, Vu, YuZhZh} and bibliographies therein.
Mainly, authors of extensions of the Hardy-Littlewood theorem impose weaker conditions on Fourier coefficients than the monotonicity condition.

In \cite{DyTi, YuZhZh},  relation \eqref{HL1} was proved for functions 
with nonnegative  Fourier coefficients
$\{a_k\}_{k=1}^{\infty}$, $\{b_k\}_{k=1}^{\infty}$ from class of general monotone sequences
defined as follows:
\begin{definition}[\cite{Ti}]
We will say that a sequence of complex numbers $\{a_k\}_{k=1}^{\infty}$
belongs to class $\textnormal{GM}$,
if there exist constants $C > 0$ and $\lambda >1$
such that,  for any $n \ge 1$,
$$
\sum_{k=n}^{2n}|a_k - a_{k+1}|
\le \frac{C}{n}\sum_{k = \frac{n}{\lambda}}^{\lambda n} |a_k|.
$$
\end{definition}

In \cite[Theorem 1.2]{DyMuTi},  equivalence \eqref{HL1} was obtained  for functions 
with real-valued  Fourier coefficients 
$\{a_k\}_{k=1}^{\infty}$, $\{b_k\}_{k=1}^{\infty}$ from  $\textnormal{GM}$ without assumption on 
nonnegativity.

In \cite[Theorem 8.4]{GrSaSa},  equivalence \eqref{HL2} was derived for functions $f(x) \sim \sum\limits_{k=0}^{\infty} c_k e^{ikx}$, 
where $\{c_k\}_{k=1}^{\infty} \in \textnormal{GM}$ is a sequence of complex  numbers 
$c_k \in S_{\alpha, \beta} = \{z \in \mathbb{C} : |\arg z -\alpha| \le \beta \}, k \ge 1$, 
$0 \le \alpha < 2\pi$, $0 \le \beta < \frac{\pi}{2}$.

In \cite[Corollary 5]{Nu},  equivalence \eqref{HL2}   
was established, in case when $2 \le p < \infty$, for functions 
$f(x) \sim \sum\limits_{k=1}^{\infty} c_k e^{ikx}$,  
where the sequence $\{c_k\}_{k=1}^{\infty}$ of complex numbers belongs to  class
of weak monotone sequences.

\begin{definition}
We will say that a sequence $\{a_{n}\}_{n=1}^{\infty}$ of complex 
numbers belongs to the class of  weak monotone sequences $\textnormal{WM}$, 
if there exists $C >  0$ such that, for any $n \ge 1$,
$$
|a_n| \le \frac{C}{n} \left|\sum_{j = 1}^{n} a_{j}\right|.
$$ 
\end{definition}  
\begin{remark}
The notions of general monotonicity and weak monotonicity were
introduced by S. Tikhonov (see \cite{Ti}, \cite{TiZe}). 
Various classes of general monotone and weak monotone sequences 
(or functions) have been studied in the context of many problems of Fourier analysis 
and approximation theory; see, for instance,  \cite {BeDyTi, LiTi}.
\end{remark}

Analogues of the Hardy-Littlewood theorem for series  
with complex coefficients 
satisfying another general monotonicity conditions can be found in \cite{BeDyTi, DyNuKa}.

Another generalizations of the Hardy-Littlewood theorem can be found in \cite{Bo1, KoNuPe, Vu}.

\subsection{Main aim}
The main aim of this paper is to define the class of sequences $\{a_k\}_{k=-\infty}^{\infty}$ of complex numbers
such that, for any $1< p < \infty$, the relation
$$
\|f\|_{L_p([-\pi,\pi])} \asymp J_p(f)
$$
holds, where $f(x) \in L_1([-\pi, \pi])$ is a function  with Fourier series $\sum\limits_{k= -\infty}^{\infty} a_k e^{ikx}$, and 
$J_p(f)$ is Paley type functional, defined as follows
$$
J_p(f):= \left(\sum_{k= -\infty}^{\infty} (|k|+1)^{p-2} |a_k|^p\right)^{\frac{1}{p}}.
$$

The similar question is considered in the setting of the following  Paley type functional 
$$
J_p^*(f) = \left(\sum\limits_{k=-\infty}^{\infty} (|k|+1)^{p-2} (a_k^*)^p\right)^{\frac{1}{p}},
$$
where $\{a_k^*\}_{k=-\infty}^{\infty}$
is symmetric nonincreasing rearrangement of sequence $\{a_k\}_{k=-\infty}^{\infty}$, 
i.e., $a_0^* \ge a_{-1}^* \ge a_1^* \ge a_{-2}^* \ge a_2^* \ge \ldots$.
\subsection{Main results}
For the sequence $\{a_k\}_{k=-\infty}^{\infty}$, for the sake of symmetry, we put 
\begin{equation*}
|\Delta a_k| := 
\begin{cases}
|a_k - a_{k+1}|, & k > 0,\\
|a_k - a_{k-1}|, & k < 0,\\
|a_0 - a_1| + |a_0 - a_{-1}|, & k = 0.
\end{cases}
\end{equation*}

We will call by the interval in $\mathbb{Z}$ any finite arithmetic progression with difference 
equal to 1.  By $\mathcal{W}$ we denote the set of all intervals in $\mathbb{Z}$. 
Let us define new two classes of general monotone sequences $\textnormal{GM}^*$ and 
$\overline{\textnormal{GM}}$.

\begin{definition}\label{large-gm-1}
We will say that a sequence $\{a_{m}\}_{m=-\infty}^{\infty}$ of complex numbers
belongs to the class $\textnormal{GM}^*$, if there exists $C > 0$ such that, for
any ${n} \ge 0$,
$$
\sum_{[2^{n-1}]\le |m| < 2^{n}}\left|\Delta a_{m}\right|\leq 
C \sup_{k  \in \mathbb{N}_0} \min (1, 2^{k-n}) \widetilde{a}_{2^k},
$$ 
where 
$$
\widetilde{a}_{2^k} = \sup_{\substack{{w} \in \mathcal{W} \\|{w}|\ge 2^{k}}} 
\frac{1}{|{w}|} \left|\sum_{{j} \in {w}} a_{j}\right|, \quad k \ge 0,
$$
and $[\cdot]$  is the floor function, and $|w|$ is the cardinality of the set $w \in \mathcal{W}$.
\end{definition}

\begin{remark}
Note that 
$$
[2^{n-1}] = 
\begin{cases}
2^{n-1}, & if \ \  n \in \mathbb{N},\\
0, & if \ \ n = 0.
\end{cases}
$$
\end{remark}

\begin{definition}\label{large-gm-2}
We will say that a sequence $\{a_{m}\}_{m=-\infty}^{\infty}$ of complex numbers
belongs to the class $\overline{\textnormal{GM}}$, if there exists $C > 0$ such that, for
any ${n} \ge 0$,
$$
\sum_{[2^{n-1}] \le |m| < 2^n}\left|\Delta a_{m}\right|\leq 
C \sup_{k  \in \mathbb{N}_0} \min (1, 2^{k-n}) \widehat{a}_{2^k},
$$
where
$$
\widehat{a}_{2^k} = \sup_{2^{k}\le |m| < 2^{k+1}} \frac{1}{|m|+1}
\left|\sum_{j=0}^{m} a_j\right|, \quad k \ge 0.
$$
\end{definition}
\begin{remark}
It is easy to see that $\overline{\textnormal{GM}} \subseteq \textnormal{GM}^*$.
\end{remark}

The main results of this paper is the following theorems.
\begin{theorem}\label{T5}
Let $1<p<\infty$ and $f \in L_1([-\pi, \pi])$ be a function with Fourier series 
$\sum\limits_{k=-\infty}^{\infty} a_{k} e^{i{kx}}$.
Let  also $a=\{a_{k}\}_{k=-\infty}^{\infty} \in \textnormal{GM}^*$,
then 
\begin{equation}\label{f4}
\|f\|_{L_p([-\pi, \pi])} \asymp J_p^*(f).
\end{equation}
\end{theorem}

\begin{theorem}\label{T6}
Let $1<p<\infty$ and $f \in L_1([-\pi, \pi])$ be a function with Fourier series 
$\sum\limits_{k=-\infty}^{\infty} a_{k} e^{i{kx}}$.
Let  also $a=\{a_{k}\}_{k=-\infty}^{\infty} \in \overline{\textnormal{GM}}$,
then 
\begin{equation}\label{f5}
\|f\|_{L_p([-\pi, \pi])} \asymp J_p(f).
\end{equation}
\end{theorem}
\subsection{Some remarks}

I. One can see that the conditions stated 
in Theorems \ref{T5} and \ref{T6} 
are presented in their most general form. 
In particular, in Theorems \ref{T5} and 
\ref{T6}, the following conditions hold:

\begin{enumerate}
\item[1.] The Fourier coefficients of the function are allowed to be complex;
\item[2.] No additional restrictions are imposed  on the coefficients, other than  the general monotonicity condition;
\item[3.] The equivalence of  function's $L_p$-norm  to the Paley type functional holds for  the entire range $1 < p < \infty$. 
\end{enumerate}

II. Note that the averages $\widetilde{a}_{2^k}$ and $\widehat{a}_{2^k}$ 
from Definitions \ref{large-gm-1} and \ref{large-gm-2} involve both 
elements of the sequence $\{a_n\}_{n=-\infty}^{\infty}$ 
with positive indices and its elements with non-positive indices.
This feature of Definitions \ref{large-gm-1} and \ref{large-gm-2} 
has a certain compensatory effect, 
which made it possible to extend the classes of sequences for which the Hardy--Littlewood theorem remains valid in a different way.
For example, in previously obtained generalizations of the Hardy--Littlewood theorem, for series of the form  
$
\sum\limits_{k=1}^{\infty} (a_k \cos kx + b_k \sin kx),
$  
general monotonicity condition is required for both sequences $\{a_k\}_{k=1}^{\infty}$ and $\{b_k\}_{k=1}^{\infty}$.  
Similarly, for series of the form  
$
\sum\limits_{k=-\infty}^{\infty} c_k e^{i kx},
$  
previously obtained generalizations require general (or weak) monotonicity condition 
for both sequences $\{c_k\}_{k=0}^{\infty}$ and $\{c_{-k}\}_{k=0}^{\infty}$.
In turn, general monotonicity conditions in classes $\textnormal{GM}^*$ 
and $\overline{\textnormal{GM}}$  allow certain parts of sequences to behave poorly.  
For instance, Section \ref{compensatory-effect} presents an example of a sequence  
$\{c_k\}_{k=-\infty}^{+\infty} \in \overline{\textnormal{GM}}$  
such that  
$\{c_k\}_{k=0}^{\infty} \notin \overline{\textnormal{GM}}$.    

III. Developing the idea of extending the classes of Fourier coefficient sequences  
for which the equivalence  $\|f\|_{L_p} \asymp J_p(f)$
remains valid, we  consider alternating series.  
For an alternating sequence $\{c_k\}_{k=-\infty}^{\infty}$,  
the condition of general monotonicity does not hold,  
even if the sequence $\{|c_k|\}_{k=-\infty}^{\infty}$ is monotone,  
because  
$
\sum\limits_{k=n}^{2n} |\Delta c_k| \asymp \sum\limits_{k=n}^{2n} |c_k|.
$  
Nevertheless, despite these restrictions,  
we have succeeded in establishing the equivalence  
$
\|f\|_{L_p} \asymp J_p(f)
$
for alternating series (see Corollary \ref{corol1}), as well as for other cases,  
by using Fourier $L_p$-multipliers.

The paper is organized as follows. 
In Section \ref{auxiliary}, we  introduce the net spaces 
and give some auxiliary results.
In Section \ref{Nursultanov-inequality}, we obtain some Fourier type inequalities
in the setting of the net spaces. 
In Section \ref{main-results}, we prove Theorems \ref{T5} and \ref{T6}. In Section \ref{comparison}, 
we compare  some classes of general monotone sequences. 
In particular, we show that  $\textnormal{GM}_{\mathbb{R}} \subsetneq \overline{\textnormal{GM}}$,
where $\textnormal{GM}_{\mathbb{R}} = 
\left\{\{a_k\}_{k=1}^{\infty} \in \textnormal{GM}: \ \ a_k \in \mathbb{R}, \ \ k \ge 1\right\}$.
In Section \ref{compensatory-effect}, 
we provide examples of sequences that illustrate 
the compensatory effect of the general monotonicity condition 
in the class $\overline{\textnormal{GM}}$.
In Section \ref{alternating-series}, we establish the equivalence 
$\|f\|_{L_p} \asymp J_p(f)$ for functions 
$f(x) \sim \sum\limits_{k=-\infty}^{\infty} c_k e^{ikx}$ satisfying condition 
$\{\lambda_k c_k\}_{k=-\infty}^{\infty} \in \overline{\textnormal{GM}}$
with some Fourier idempotent multiplier $\{\lambda_k\}_{k=-\infty}^{\infty}$.  

Throughout the paper, the expression $L \lesssim R$ 
means that $L \le CR $ for some constant $C > 0$. 
Moreover, $L \asymp R$ stands for $L \lesssim R \lesssim L$.

\section{Auxiliary results}\label{auxiliary}

Let $(\Omega, \mu)$ be a measurable space, and let $f(x)$ be a $\mu$-measurable on $\Omega$ function. By
$f^{*}(t)$ we denote the nonincreasing rearrangement of $f(x)$, i.e.,
$$
f^*(t) = \inf\{\sigma: \mu\{x \in \Omega: |f(x)|> \sigma\}\le t\}.
$$ 

For $0 < p \le \infty$, $0 < q \le \infty$, the Lorentz space $L_{p,q}(\Omega)$
is the set of $\mu$-measurable functions for which, the functional
$$
\|f\|_{L_{p,q}} = 
\begin{cases}
\left(\int\limits_0^{\mu(\Omega)} t^{\frac{q}{p}-1} \left(f^{*}(t)\right)^q  dt\right)^{\frac{1}{q}}, & \text{for} \ 0 < p <\infty \ \text{and} \ 0 < q < \infty, \\
\sup\limits_{t \in [0,\, \mu(\Omega)]} t^{\frac{1}{p}}f^*(t), & \text{for} \ 0 < p \le \infty \ \text{and} \ q = \infty,
\end{cases}
$$
is finite.

\begin{remark}
\begin{enumerate}
\item[\textnormal{1)}] In case when $\Omega = \mathbb{Z}$ and $\mu$ is a counting measure, 
corresponding discrete Lorentz spaces are denoted as  $l_{p,q}(\mathbb{Z})$.
\item[\textnormal{2)}] In case when $p=q$, $L_{p,p}(\Omega) = L_p(\Omega)$.
\end{enumerate}
\end{remark}

Now, let's describe the main tool of this paper: net spaces.
The net spaces was introduced by E. Nursultanov in \cite{Nu}. 
Net space methods have been used to solve many problems 
in Fourier analysis (see \cite{DyNu, DyNuKa, DyNuTiWe, Nu, 
Nu98East, NuTi}), 
approximation theory (see \cite{Nu06, NuTl}), and Fourier multiplier theory 
(see \cite{TlNuBa}).

To define the discrete net spaces we need introduce the average of the sequence. 

For any sequence $\{a_{m}\}_{m=-\infty}^{\infty}$  of
complex numbers, we will define the average $\{\widetilde{a}_{k}\}_{k=1}^{\infty}$ over $\mcw$ 
as follows: 
$$
{\widetilde{a}}_{k} = {\widetilde{a}}_{k}(\mcw)=\sup_{\substack{{w}  \in \mcw\\ |w|\ge k}}
\frac{1}{|w|}\left|\sum_{m \in w} a_{m}\right|, \quad k \in \mathbb{N}.
$$
\begin{definition}[\cite{Nu}]
Let  $1 < p < \infty$, $1 \le q \le \infty$.
Discrete net space $n_{p,q}$ is the set of sequences of complex numbers 
$\{a_{m}\}_{m=-\infty}^{\infty}$  for which, the functional
$$
\|a\|_{n_{p,q}} = 
\begin{cases}\left(\sum\limits_{k=1}^\infty 
 k^{\frac qp-1}{\widetilde{a}}_{k}^q\right)^{1/q}, & \textnormal{for} \ 1 < p < \infty \ \textnormal{and} \ 1 < q < \infty,\\
\sup\limits_{k \ge 1} k^{\frac{1}{p}} \widetilde{a}_k,& \textnormal{for} \ 1 < p \le \infty \ \textnormal{and} \   q  = \infty,
\end{cases}
$$
is finite.
\end{definition}

\begin{lemma}\label{alm_mon}
Let $\{a_n\}_{n=-\infty}^{\infty}$ be a sequence such that $\widetilde{a}_1 < \infty$. Then, for any integer $k \ge 0$,
\begin{equation}\label{almost_mon}
\widetilde{a}_{2^k} \le 5 \widetilde{a}_{2^{k+1}}.
\end{equation}
\end{lemma}
\begin{proof}
Fix an integer $k \ge 0$ and consider an interval 
$w_1 \subset \mathbb{Z}$ with $|w_1| \ge 2^k$. 
Define an interval $w \subset \mathbb{Z}$ satisfying the following conditions:
\begin{enumerate}
\item[(i)] $w_1 \subset w$;
\item[(ii)] $|w| = 3 |w_1|$;
\item[(iii)] the set $w_2 := w \setminus w_1$ is an interval of integer numbers.
\end{enumerate}
For these  intervals, observe that
$$
\left|\sum_{m \in w_1} a_m\right| = \left|\sum_{m \in w} a_m - \sum_{m \in w_2} a_m\right|
\le \left|\sum_{m \in w} a_m\right| + \left|\sum_{m \in w_2} a_m\right|.
$$
Dividing both sides of this inequality by $|w_1|$, we obtain
\begin{equation*}
\begin{split}
\frac{1}{|w_1|}\left|\sum_{m \in w_1} a_m\right| &  
\le \frac{1}{|w_1|}\left|\sum_{m \in w} a_m\right| + 
\frac{1}{|w_1|}\left|\sum_{m \in w_2} a_m\right| 
= \frac{3}{|w|}\left|\sum_{m \in w} a_m\right| + 
\frac{2}{|w_2|}\left|\sum_{m \in w_2} a_m\right|.
\end{split}
\end{equation*}
Taking the supremum over all intervals $w_1$ with $|w_1| \ge 2^k$, we get
\begin{equation*}
\begin{split}
\widetilde{a}_{2^k} & = \sup_{|w_1|\ge 2^k}\frac{1}{|w_1|}\left|\sum_{m \in w_1} a_m\right|   
\le \sup_{|w_1|\ge 2^k}\left( \frac{3}{|w|}\left|\sum_{m \in w} a_m\right| + 
\frac{2}{|w_2|}\left|\sum_{m \in w_2} a_m\right|\right) \\
& \le \sup_{|w|\ge 3\cdot 2^k} \frac{3}{|w|}\left|\sum_{m \in w} a_m\right| + 
\sup_{|w_2|\ge 2\cdot 2^k}\frac{2}{|w_2|}\left|\sum_{m \in w_2} a_m\right| 
 \le \sup_{|w_2|\ge  2^{k+1}}\frac{5}{|w_2|}\left|\sum_{m \in w_2} a_m\right| = 5 \widetilde{a}_{2^{k+1}}.
\end{split}
\end{equation*}
Thus, the desired inequality is established.
\end{proof}

Using standard arguments, along with the monotonicity of the sequence 
$\{\widetilde{a_k}\}_{k=1}^{\infty}$ and inequality \eqref{alm_mon}, 
we arrive at the following lemma.   
\begin{lemma}
For any $1 < p < \infty$, $1 \le q \le \infty$, 
the following equivalence holds:
$$
\|a\|_{n_{pq}} \asymp \left(\sum_{k=0}^{\infty} \left(2^{\frac{k}{p}}\widetilde{a}_{2^{k}}\right)^q\right)^{\frac{1}{q}}.
$$
\end{lemma}

\begin{lemma}\label{zero_modif}
Let $\{a_k\}_{k=-\infty}^{\infty}$ be a sequence of complex numbers, 
and define a new sequence $\{b_k\}_{k=-\infty}^{\infty}$ as follows:
$$
b_k = 
\begin{cases}
a_k, & k \neq 0;\\
\frac{1}{2}a_0, & k = 0.
\end{cases}
$$ 
Then, for any integer $r \geq 0$, the following inequality holds:
$$
\widetilde{b}_{2^r}
\le 6\widetilde{a}_{2^r}.
$$
\end{lemma}
\begin{proof}
Let $r \geq 0$, and let $w \subset \mathbb{Z}$ 
be an interval of integers such that $|w| \geq 2^r$, 
$|w| > 1$, and $0 \in w$. Then, we have
\begin{equation}\label{ab-1}
\begin{split}
\frac{1}{|w|} \left|\sum_{m\in w} b_m\right| 
& = \frac{1}{|w|} \left|b_0 + \sum_{m\in w \setminus{\{0\}}} b_m\right| 
= \frac{1}{|w|} \left|\frac{1}{2}a_0 + \sum_{m\in w \setminus{\{0\}}} a_m\right| \\
& = \frac{1}{2|w|} \left|a_0 + 2\sum_{m\in w \setminus{\{0\}}} a_m\right|
\le \frac{1}{2|w|} \left|\sum_{m\in w } a_m\right|
+ \frac{1}{2|w|} \left|\sum_{m\in w \setminus{\{0\}}} a_m\right|\\
& = \frac{1}{2|w|} \left|\sum_{m\in w } a_m\right|
+ \frac{|w|-1}{2|w|} \frac{1}{|w|-1} \left|\sum_{m\in w \setminus{\{0\}}} a_m\right|\\
& < \frac{1}{|w|} \left|\sum_{m\in w } a_m\right|
+  \frac{1}{|w|-1} \left|\sum_{m\in w \setminus{\{0\}}} a_m\right|.
\end{split}
\end{equation} 
Now, consider the family of intervals $\Omega_r = \{w \in \mathcal{W} \mid |w| \geq 2^r\}$. 
Divide this family into two subfamilies: $\Omega_{r,1} = \{w \in \Omega_r \mid 0 \in w\}$ and $\Omega_{r,2} = \{w \in \Omega_r \mid 0 \notin w\}$.

Let $r \ge 1$, then inequalities \eqref{almost_mon}  
and \eqref{ab-1} imply that 
\begin{equation*}
\begin{split}
\widetilde{b}_{2^r} & =  
\sup_{w \in \Omega_r}\frac{1}{|w|}\left|\sum_{m \in w} b_m\right| = 
\max\left\{\sup_{w\in \Omega_{r,1}}\frac{1}{|w|}\left|\sum_{m \in w} b_m\right|,
\sup_{w\in \Omega_{r,2}}\frac{1}{|w|}\left|\sum_{m \in w} b_m\right|\right\} \\
& \le \max\left\{\sup_{w\in \Omega_{r,1}}\left(\frac{1}{|w|} \left|\sum_{m\in w } a_m\right|
+  \frac{1}{|w|-1} \left|\sum_{m\in w \setminus{\{0\}}} a_m\right|\right),
\sup_{w\in \Omega_{r,2}}\frac{1}{|w|}\left|\sum_{m \in w} a_m\right|\right\}\\
& \le \max\left\{\sup_{w\in \Omega_{r}}\frac{1}{|w|} \left|\sum_{m\in w } a_m\right|
+  \sup_{w\in \Omega_{r-1}}\frac{1}{|w|} \left|\sum_{m\in w } a_m\right|,
\sup_{w\in \Omega_{r}}\frac{1}{|w|}\left|\sum_{m \in w} a_m\right|\right\}\\
& \le \max\left\{\sup_{w\in \Omega_{r}}\frac{1}{|w|} \left|\sum_{m\in w } a_m\right|
+  5\sup_{w\in \Omega_{r}}\frac{1}{|w|} \left|\sum_{m\in w } a_m\right|,
\sup_{w\in \Omega_{r}}\frac{1}{|w|}\left|\sum_{m \in w} a_m\right|\right\} = 6\widetilde{a}_{2^r}.
\end{split}
\end{equation*}
By similar arguments, we obtain $\widetilde{b}_1 \le 2\widetilde{a}_1$.
This completes the proof.
\end{proof}


The following lemma follows from \cite[Assertions 1, 2]{Nu}.
\begin{lemma}\label{netlesslorentz}
Let $1 < p \le \infty$, $1 < q \le \infty$. Then
$$
\|a\|_{n_{p,q}} \lesssim \|a\|_{l_{p,q}}.
$$
\end{lemma}

We also need the following known Hardy type inequalities 
(see, for instance, \cite[Lemma 2.5]{PoSiTi}).
\begin{lemma}[Hardy inequalities]
Let $1 < q < \infty$, $0 < \alpha < 1$ and let $\{a_n\}_{n=0}^{\infty}$ be a nonnegative sequence.
Then 
\begin{equation}\label{Hardy1}
\left(\sum_{k=0}^{\infty} \left(2^{\alpha k} \sum_{m=k}^{\infty} a_m\right)^q\right)^{\frac{1}{q}}
\lesssim \left(\sum_{k=0}^{\infty} \left(2^{\alpha k} a_k\right)^q\right)^{\frac{1}{q}},
\end{equation}
\begin{equation}\label{Hardy2}
\left(\sum_{k=0}^{\infty} \left(2^{(\alpha - 1) k} \sum_{m=0}^{k} 2^m a_m\right)^q\right)^{\frac{1}{q}}
\lesssim \left(\sum_{k=0}^{\infty} \left(2^{\alpha k} a_k\right)^q\right)^{\frac{1}{q}}.
\end{equation}
\end{lemma}

\section{The estimates of the net spaces norms of the Fourier coeffients and their applications}\label{Nursultanov-inequality}

The following Fourier inequality for the case  $p \ge 2$ 
was established in \cite[Theorem 3]{Nu}. 
For the reader's convenience, we provide a full proof of this Theorem.
\begin {theorem}\label{T1}
Let $ 1<p<\infty$, $p'=p/(p-1)$ and  $f\in L_p([-\pi, \pi])$ be a function
with Fourier series
 $\sum\limits_{k=-\infty}^{\infty} a_{k} e^{ikx}$. Then  
\begin{equation}\label{nurs-ineq}
\|a\|_{n_{ p',p}} \lesssim 
\|f\|_{L_p([-\pi, \pi])}.
\end{equation}
\end {theorem}
\begin{proof}
For an arbitrary interval  $w \subset \mathbb{Z}$, 
let  $D_w(x) = \sum\limits_{m \in w} e^{-imx}$
denote the corresponding Dirichlet kernel.
 
Let $1 < q < \infty$, then the standard calculation gives
\begin{equation*}
\begin{split}
\frac{1}{|w|^{\frac{1}{q}}} \left|\sum_{m \in w} a_m\right|  & = 
\frac{1}{2\pi|w|^{\frac{1}{q}}} \left|\sum_{m \in w} \int_{-\pi}^{\pi} f(x) e^{-imx} dx\right|\\
& \le \frac{1}{2\pi|w|^{\frac{1}{q}}} \int_{-\pi}^{\pi} |f(x)||D_w(x)|dx 
\le \frac{1}{2\pi|w|^{\frac{1}{q}}} \int_{-\pi}^{\pi} |f(x)|\left|\frac{2\sin \frac{|w| x}{2}}{\sin  \frac{x}{2}}\right|dx\\ 
& \lesssim \frac{1}{|w|^{\frac{1}{q}}} \int_{-\pi}^{\pi} |f(x)|\min \left(|w|, \frac{1}{|x|}\right) dx
= \int_{-\pi}^{\pi} |f(x)|\varphi_w(x)dx. 
\end{split}
\end{equation*}
Let us estimate function $\varphi_w(x)$. If $|w| \ge \frac{1}{|x|}$, then
$$
\varphi_w(x) = \frac{1}{|x| |w|^{\frac{1}{q}}} =
\frac{1}{|x|^{\frac{1}{q'}} (|x||w|)^{\frac{1}{q}}} \le \frac{1}{|x|^{\frac{1}{q'}}}.   
$$
If $|w| < \frac{1}{|x|}$, then
$$
\varphi_{w}(x) = \frac{|w|}{|w|^{\frac{1}{q}}} = |w|^{\frac{1}{q'}} \le \frac{1}{|x|^{\frac{1}{q'}}}.
$$
Therefore, by using the rearrangement inequality, we get
\begin{equation}\label{in1}
\frac{1}{|w|^{\frac{1}{q}}} \left|\sum_{m \in w} a_m\right|
\le \int_{-\pi}^{\pi} |f(x)| |x|^{-\frac{1}{q'}} dx \le \int_0^{2\pi} f^*(t) t^{-\frac{1}{q'}} dt = \|f\|_{L_{q,1}([-\pi,\pi])}.
\end{equation}
On the other hand, we have
\begin{equation}\label{in2}
\|a\|_{n_{q', \infty}} = \sup_{k \in \mathbb{N}} k^{\frac{1}{q'}} \widetilde{a}_k \le
 \sup_{k \in \mathbb{N}} \sup_{\stackrel{|w| \ge k}{w \in \mcw}} \frac{1}{|w|^{\frac{1}{q}}}\left|\sum_{m \in w} a_m\right|
= \sup_{w \in \mcw}  \frac{1}{|w|^{\frac{1}{q}}}\left|\sum_{m \in w} a_m\right|.
\end{equation}
Combining inequalities \eqref{in1}, \eqref{in2}, we conclude that 
\begin{equation}\label{in3}
\|a\|_{n_{q', \infty}} \le \|f\|_{L_{q,1}([-\pi,\pi])}.
\end{equation}
Consider operator $T: f \mapsto  \{a_k\}_{k \in \mathbb{Z}}$, 
where $a_k = \frac{1}{2\pi}\int_{-\pi}^{\pi} f(x) e^{-ikx} dx$. Inequality \eqref{in3}
implies the boundedness 
$$
T : L_{q,1}([-\pi,\pi]) \to n_{q', \infty},  \quad 1 <  q <  \infty.
$$
Let now $p \in (1, \infty)$, then for $1 < p_0 < p < p_1 < \infty$, we have the boundedness 
$$
T : L_{p_i,1}([-\pi,\pi]) \to n_{p_i^{'}, \infty}, \quad i = 0,1.
$$
By interpolation, using Marcinkiewicz's theorem \cite[Section 5.3]{BeLo} and the interpolation
theorem of the net spaces \cite[Theorem 1]{Nu}, we obtain the boundedness
$$
T : L_{p, \tau}([-\pi,\pi]) \to n_{p', \tau},  \quad 1 < \tau < \infty,
$$
i.e.,
$$
\|a\|_{n_{p', \tau}} \le C\|f\|_{L_{p, \tau}([-\pi,\pi])}, \quad 1 < \tau < \infty,
$$
in particular, this gives the estimate
$$
\|a\|_{n_{p', p}} \le C\|f\|_{L_{p}([-\pi,\pi])}.
$$
\end{proof}

\begin{remark}
Inequality \eqref{nurs-ineq} was first obtained by E. Nursultanov in \cite[Theorem 3]{Nu}.
The main advantage of this inequality, compared to other Fourier inequalities, 
is  its validity for all $1 < p < \infty$ without imposing additional 
conditions  on the function or  its Fourier coefficients.
Later, similar inequalities were derived in many papers; 
see, for instance, \cite{DyNu, DyNuKa, DyNuTiWe, KoNuPe, Nu98East, NuTi}.

A weaker form of this inequality was earlier obtained 
by Y. Sagher in \cite[Theorem 2.4]{Sa} (see also \cite{DyMuTi, GrSaSa, Og}).


In the case when $1 < p < 2$, inequality \eqref{nurs-ineq} follows from the inequality  
$\|a\|_{l_{p'p}} \lesssim \|f\|_{L_p}$ (see \cite[Theorem 4.3]{Hu}) and Lemma \ref{netlesslorentz}.
  
\end{remark}

Let $1 < p < \infty$, for an integrable function $f(x)\sim \sum\limits_{k=-\infty}^{\infty} a_{k} e^{ikx}$,
define
$$
I_p(f) := \left(\sum_{k=0}^\infty \left(2^{\frac{k}{p'}}\Theta_{k}(f)\right)^p\right)^{1/p},
$$
where
$$
\Theta_{k}(f) = \sum\limits_{[2^{k-1}] \le |m| < 2^{k}}\left|\Delta a_{m}\right|,  \quad k \ge 0.
$$

\begin{lemma}\label{L1}
Let  $f({x})$ be an integrable on $[-\pi,\pi]$ function with Fourier series     
$\sum\limits_{k=-\infty}^{\infty} a_{k} e^{ikx}$.
Let also
$$
S_N(f) =\sum_{m=-2^N}^{2^N} a_{m}e^{imx}, \quad N \in \mathbb{N}.
$$
Then, for any $1 < p < \infty$ and $N \in \mathbb{N}$,
$$
I_p(S_N(f)) \le 2^{\frac{1}{p'}}I_p(f).
$$
\end{lemma}
\begin{proof} 

For any $N\in  \mathbb{N}$, we get
\begin{equation*}
\begin{split}
& \left(I_p(S_N(f))\right)^p  = \sum_{k=0}^{N} \left(2^{\frac{k}{p'}}\Theta_{k}(f)\right)^p 
+ \left(2^{\frac{N+1}{p'}}\left(\left| a_{2^N}\right| + \left| a_{-2^N}\right|\right)\right)^p \\
& =
\sum_{k=0}^{N} \left(2^{\frac{k}{p'}}\Theta_{k}(f)\right)^p 
+ \left(2^{\frac{N+1}{p'}}\left(\left| \sum_{k=2^N}^\infty\Delta a_k\right| + 
\left| \sum_{k=-\infty}^{-2^N}\Delta a_k\right|\right)\right)^p\\ 
& \le \sum_{k=0}^{N} \left(2^{\frac{k}{p'}}\Theta_{k}(f)\right)^p 
+ \left(2^{\frac{N+1}{p'}}\left( \sum_{k=2^N}^\infty\left|\Delta a_k\right| + 
 \sum_{k=-\infty}^{-2^N}\left|\Delta a_k\right|\right)\right)^p \\
& = \sum_{k=0}^{N} \left(2^{\frac{k}{p'}}\Theta_{k}(f)\right)^p 
+ \left(2^{\frac{N+1}{p'}}\left( \sum_{k=N}^\infty\sum_{m=2^k}^{2^{k+1}-1}\left|\Delta a_m\right| + 
 \sum_{k=N}^{\infty}\sum_{m=-(2^{k+1} -1)}^{-2^k}\left|\Delta a_m\right|\right)\right)^p \\
& =
\sum_{k=0}^{N} \left(2^{\frac{k}{p'}}\Theta_{k}(f)\right)^p 
+ \left(2^{\frac{N+1}{p'}} \sum_{k=N+1}^\infty \Theta_k(f)\right)^p. \\
\end{split} 
\end{equation*}

Using H\"older's inequality, we obtain
\begin{equation*}
\begin{split}
\left(I_p(S_N(f))\right)^p   & \le  \sum_{k=0}^{N} \left(2^{\frac{k}{p'}}\Theta_{k}(f)\right)^p + \\
& + \left(2^{\frac{N+1}{p'}} \left(\sum_{k=N+1}^\infty \left(2^{\frac{k}{p'}}\Theta_k(f)\right)^p\right)^{\frac{1}{p}}\left(\sum_{k=N+1}^\infty \left(2^{-\frac{k}{p'}}\right)^{p'}\right)^{\frac{1}{p'}}\right)^p\\
& = \sum_{k=0}^{N} \left(2^{\frac{k}{p'}}\Theta_{k}(f)\right)^p 
+ 2^{\frac{p}{p'}} \sum_{k=N+1}^\infty \left(2^{\frac{k}{p'}}\Theta_k(f)\right)^p
 \le 2^{\frac{p}{p'}} \sum_{k=0}^{\infty} \left(2^{\frac{k}{p'}}\Theta_{k}(f)\right)^p.\\
\end{split}
\end{equation*}

Therefore,
$$
I_p(S_N(f)) \le 2^{\frac{1}{p'}}I_p(f).
$$
\end{proof}
\begin{theorem}\label{T4}
Let $1<p<\infty$, and $f \in L_1([-\pi,\pi])$ with Fourier series $\sum\limits_{k=-\infty}^{\infty} a_{k} e^{ikx}$.
Let
$$
I_p(f) < \infty,
$$
then $f\in L_p([-\pi, \pi])$  and, moreover, the following inequality holds:
\begin{equation}\label{f2}
\|f\|_{L_p} \lesssim I_p(f).
\end{equation}
\end{theorem}

\begin{proof} First, consider case where $f$ 
is a trigonometric polynomial, i.e.,  
$f({x})= \sum\limits_{k=-\infty}^{\infty} a_{k} e^{ikx}$, 
where only a finite number of coefficients are nonzero. 
Using Parseval's equality 
\begin{equation}\label{in11}
\|f\|_{L_p} =\sup_{\|g\|_{L_{p'}}=1}\left|\int_{[-\pi,\pi]}f({x})\overline{g({x})}{dx} \right|= 
\sup_{\|g\|_{L_{p'}}=1}\left|\sum_{k=-\infty}^{\infty} a_{k}\overline{b}_{k}\right|,
\end{equation}
where $\{b_{k}\}_{k=-\infty}^{\infty}$ is the sequence of   Fourier coefficients of the function $g$,   
$\overline{g({x})}$ is the complex conjugate  of $g(x)$,  
$\{\overline{b}_{k}\}_{k=-\infty}^{\infty}$ is the complex conjugate 
of the  sequence $\{b_{k}\}_{k=-\infty}^{\infty}$.
Let us define
$$
d_k = 
\begin{cases}
b_k, & \text{for} \ k \neq 0;\\
\frac{1}{2}b_0, & \text{for} \ k = 0.
\end{cases}
$$

Since the sequence $\{a_k\}_{k=-\infty}^{\infty}$ contains only  finitely 
many nonzero elements, we have
\begin{equation*}
\begin{split}
\sum_{k=-\infty}^{\infty} a_{k}\overline{b}_{k} & =
\sum_{k=-\infty}^{-1} a_{k}\overline{b}_{k} + \frac{1}{2}a_0\overline{b}_0 + 
\frac{1}{2}a_0\overline{b}_0 + \sum_{k=1}^{\infty} a_{k}\overline{b}_{k}
= \sum_{k=-\infty}^{0} a_{k}\overline{d}_{k} +  
\sum_{k=0}^{\infty} a_{k}\overline{d}_{k}\\
& =\sum_{k=-\infty}^{0}\overline{d}_{k}\sum_{m= -\infty}^{k}(a_m - a_{m-1}) 
+ \sum_{k=0}^{\infty}\overline{d}_{k}\sum_{m=k}^{\infty}(a_m - a_{m+1})=\\
& =\sum_{m=-\infty}^{0}(a_{m} - a_{m-1}) \overline{\sum_{k=m}^{0}d_k}
+ \sum_{m=0}^{\infty}(a_{m} - a_{m+1}) \overline{\sum_{k=0}^{m}d_k}. \\
\end{split}
\end{equation*}

Therefore,
\begin{equation*}
\begin{split}
\left|\sum_{k=-\infty}^{\infty} a_{k}\overline{b}_{k}\right| 
&\le \sum_{m=-\infty}^{0}|a_{m} - a_{m-1}| \left|\sum_{k=m}^{0}d_k\right|
+ \sum_{m=0}^{\infty}|a_{m} - a_{m+1}| \left|{\sum_{k=0}^{m}d_k}\right|\\
& \le \sum_{r=0}^\infty\sum_{m=-(2^{r}-1)}^{- [2^{r-1}]}\left|\Delta a_{m}\right|\left|\sum_{k = m}^{0}d_{k}\right| + 
 \sum_{r=0}^\infty\sum_{m=[2^{r-1}]}^{2^{r}-1}\left|\Delta a_{m}\right|\left|\sum_{k=0}^m d_{k}\right|\\
& =  \sum_{r=0}^\infty\sum_{m=-(2^{r}-1)}^{- [2^{r-1}]}(|m|+1)\left|\Delta a_{m}\right|\frac{1}{|m|+1}\left|\sum_{k = m}^{0}d_{k}\right| \\
&+  \sum_{r=0}^\infty\sum_{m=[2^{r-1}]}^{2^{r}-1}(|m|+1)\left|\Delta a_{m}\right|\frac{1}{|m|+1}\left|\sum_{k=0}^m d_{k}\right|\\
& \le \sum_{r=0}^\infty\sum_{m=-(2^{r}-1)}^{- [2^{r-1}]}(|m|+1)\left|\Delta a_{m}\right|\widetilde{d}_{|m|+1} + 
\sum_{r=0}^\infty\sum_{m=[2^{r-1}]}^{2^{r}-1}(|m|+1)\left|\Delta a_{m}\right|\widetilde{d}_{|m|+1}.
\end{split}
\end{equation*}
Using monotonicity of the sequence 
$\{\widetilde{d}_m\}_{m=1}^{\infty}$ and inequality \eqref{alm_mon}, we obtain

\begin{equation*}
\begin{split}
\left|\sum_{k=-\infty}^{\infty} a_{k}\overline{b}_{k}\right| & 
\le  \sum_{r=0}^\infty 2^r\widetilde{d}_{[2^{r-1}]+1} \sum_{[2^{r-1}] \le |m| < 2^{r}-1}\left|\Delta a_{m}\right|\\
& \le \widetilde{d}_{1} |\Delta a_0| + 5\sum_{r=1}^\infty 2^r\widetilde{d}_{2^{r}} \sum_{[2^{r-1}] \le |m| < 2^{r}-1}\left|\Delta a_{m}\right|\\
\\
& \le 5\sum_{r=0}^\infty 2^r\widetilde{d}_{2^{r}} \sum_{[2^{r-1}] \le |m| < 2^{r}-1}\left|\Delta a_{m}\right|.
\end{split}
\end{equation*}
By using Lemma \ref{zero_modif}, we derive
$$
\left|\sum_{k=-\infty}^{\infty} a_{k}\overline{b}_{k}\right| 
 \lesssim  \sum_{r=0}^\infty 2^{r} \widetilde{b}_{2^{r}} \Theta_{r}(f). 
$$

Applying H\"older's inequality and Theorem \ref{T1}, we get
\begin{equation}\label{in13}
\begin{split}
\left|\sum_{k=-\infty}^{\infty} a_{k}\overline{b}_{k}\right| & \le
\sum_{r=0}^\infty 2^{\frac{r}{p}} \widetilde{b}_{2^{r}} \cdot 2^{\frac{r}{p'}}\Theta_{r}(f)  
\le \left(\sum_{r=0}^\infty \left(2^{\frac{r}{p}} \widetilde{b}_{2^{r}}\right)^{p'}  \right)^{\frac{1}{p'}}
\left(\sum_{r=0}^\infty  \left(2^{\frac{r}{p'}}\Theta_{r}(f)\right)^{p} \right)^{\frac{1}{p}} \\
&
\lesssim \|b\|_{n_{pp'}}\cdot I_p(f) \lesssim \|g\|_{L_{p'}}\cdot I_p(f). 
\end{split}
\end{equation}  
Combining \eqref{in11} and \eqref{in13}, we conclude
$$
\|f\|_{L_p} \lesssim I_p(f).
$$
Now consider general case where $f$ satisfies  conditions
of Theorem. Let
$$
S_N(f)=\sum_{k=-2^N}^{2^N}a_{k}e^{ikx}.
$$

From the above, we know that, for any $N \in \mathbb{N}$, 
$$
\|S_N(f)\|_{L_p} \lesssim I_p(S_N(f)).
$$
From Lemma \ref{L1} it follows that, for any $N \in \mathbb{N}$,
$$
I_p(S_N(f))\le 2^{\frac{1}{p'}}I_p(f).
$$
Since, for any $1 < p <\infty$, $\lim\limits_{N\to \infty} \|f - S_N(f)\|_{L_p} = 0$,
we get
$$
\|f\|_{L_p([-\pi,\pi])} \lesssim I_{p}(f).
$$ 
\end{proof}
\begin{remark}
The estimate \eqref{f2} in the non-periodic case 
was obtained in \cite[Theorem 1]{KoNuPe1}.  
An estimate of this type was obtained in 
\cite[Theorem 3]{DyTiJMAA} (see also \cite[Theorem 2.2]{DyTiWaterman}). 
\end{remark}

\section{The proof of the main results}\label{main-results}

\begin{proof}[Proof of Theorem \ref{T5}]
By using the  condition of general monotonicity in $\textnormal{GM}^*$, we derive 
\begin{equation*}
\begin{split}
I_p(f) & =\left(\sum_{k=0}^{\infty} \left(2^{\frac{k}{p'}}\Theta_{k}(f)\right)^p\right)^{1/p}
 \lesssim \left(\sum_{k=0}^{\infty}\left(2^{\frac{k}{p'}}\sup_{l  \in \mathbb{N}_0} \min (1, 2^{l-k}) \widetilde{a}_{2^l}\right)^p\right)^{1/p} \\
& \lesssim \left(\sum_{k=0}^{\infty}\left(2^{\frac{k}{p'}}\sup_{0\le l \le k} 2^{l-k}\widetilde{a}_{2^l}\right)^p\right)^{1/p}
+ \left(\sum_{k=0}^{\infty}\left(2^{\frac{k}{p'}}\sup_{l\ge k + 1} \widetilde{a}_{2^l}\right)^p\right)^{1/p} =: I_1 + I_2.
\end{split}
\end{equation*}
Let's estimate $I_1$. By using Hardy's inequality \eqref{Hardy2}
\begin{equation*}
\begin{split}
I_1 & = \left(\sum_{k=0}^{\infty}\left(2^{\frac{k}{p'}}\sup_{0\le l \le k} 2^{l-k}\widetilde{a}_{2^l}\right)^p\right)^{1/p}
\le \left(\sum_{k=0}^{\infty}\left(2^{\left(\frac{1}{p'}-1\right)k}\sum_{l=0}^k 2^l \widetilde{a}_{2^l}\right)^p\right)^{1/p}\\
&  \lesssim 
\left(\sum_{k=0}^{\infty}\left(2^{\frac{k}{p'}}\widetilde{a}_{2^k}\right)^p\right)^{1/p}
\asymp \|a\|_{n_{p',p}}.
\end{split}
\end{equation*}
Let's estimate $I_2$. Since $\{\widetilde{a}_{2^k}\}_{k=0}^{\infty}$ 
is a nonincreasing sequence, we obtain
$$
I_2 = 
\left(\sum_{k=0}^{\infty}\left(2^{\frac{k}{p'}}\sup_{l\ge k+1} \widetilde{a}_{2^l}\right)^p\right)^{1/p}
\le  \left(\sum_{k=0}^{\infty}\left(2^{\frac{k}{p'}}\widetilde{a}_{2^{k+1}}\right)^p\right)^{1/p}
\lesssim \|a\|_{n_{p', p}}.
$$
Therefore, by using Theorem \ref{T1}
$$
I_p(f) \lesssim I_1 + I_2 \lesssim \|a\|_{n_{p',p}}\lesssim \|f\|_{L_p([-\pi,\pi])}.
$$
Applying Theorem \ref{T4}, we get
\begin{equation}\label{ip}
I_p(f) \asymp \|f\|_{L_p([-\pi,\pi])}.
\end{equation}

Now, we show that for any $1 < p < \infty$ and for arbitrary $\{a_k\}_{k=-\infty}^{\infty}$,
inequality
$J_p^*(f) \lesssim I_p(f)$ holds.  

Repeating the same arguments as in \eqref{in11} and \eqref{in13}
and using Lemma \ref{netlesslorentz},  
we obtain 
\begin{equation*}
\begin{split}
J_p^*(f) & \asymp \|a\|_{l_{p'p}} \asymp  \sup_{\|b\|_{l_{pp'}}=1}\left|\sum_{k=-\infty}^{\infty}a_k b_k\right| 
 \lesssim \sup_{\|b\|_{l_{pp'}}=1}\|b\|_{n_{pp'}}I_p(f) \\
& \lesssim \sup_{\|b\|_{l_{pp'}}=1}\|b\|_{l_{pp'}}I_p(f) = I_p(f). 
\end{split}
\end{equation*}

On the other hand,  since $\{a_k\}_{k=-\infty}^{\infty} \in \textnormal{GM}^*$,
\begin{equation*}
\begin{split}
I_p(f)  & = \left(\sum_{k=0}^{\infty} \left(2^{\frac{k}{p'}}\Theta_{k}(f)\right)^p\right)^{1/p}
 \lesssim \left(\sum_{k=0}^{\infty}\left(2^{\frac{k}{p'}}\sup_{l  \in \mathbb{N}_0} \min (1, 2^{l-k}) \widetilde{a}_{2^l}\right)^p\right)^{1/p} \\
& \lesssim \|a\|_{n_{p',p}} \lesssim \|a\|_{l_{p',p}} \asymp  J_p^*(f).
\end{split}
\end{equation*}
Therefore,
\begin{equation}\label{jps}
I_p(f) \asymp  J_p^*(f).
\end{equation}
Combining relations \eqref{ip} and \eqref{jps} we prove Theorem \ref{T5}. 
\end{proof}

\begin{remark}
From the proof of Theorem \ref{T5}, it follows that for function $f(x)$ 
whose Fourier coefficients 
$\{a_k\}_{k = -\infty}^{\infty}$
belong to the class $\textnormal{GM}^*$, the following equivalence holds:
$$
\|f\|_{L_p([-\pi,\pi])} \asymp I_p(f) = 
\left(\sum_{k=0}^\infty \left(2^{\frac{k}{p'}} \sum\limits_{[2^{k-1}] \le |m| < 2^{k}}\left|\Delta a_{m}\right| \right)^p\right)^{1/p}, \quad 1 < p < \infty.
$$
\end{remark}

\begin{proof}[Proof of Theorem \ref{T6}]
For any $1 < p < \infty$, we define
$$
A_p(a) := \left(\sum_{k=0}^{\infty} \left(2^{\frac{k}{p'}} \sup_{r \in \mathbb{N}_0}\min(1,2^{r-k})\widehat{a}_{2^r}\right)^p\right)^{\frac{1}{p}}.
$$
Then, the following inequality holds:
$$
A_p(a) \lesssim
\left(\sum_{k=0}^{\infty} \left(2^{\frac{k}{p'}} \sup_{0 \le r \le k} 2^{r-k}\widehat{a}_{2^r}\right)^p\right)^{\frac{1}{p}}
+ \left(\sum_{k=0}^{\infty} \left(2^{\frac{k}{p'}} \sup_{r \ge k + 1} \widehat{a}_{2^r}\right)^p\right)^{\frac{1}{p}}
=: I_1 + I_2.
$$
Let's estimate $I_1$. For any $k \in \mathbb{N}_0$, we have
\begin{equation*}
\begin{split}
\sup_{0 \le r \le k} 2^{r-k}\widehat{a}_{2^r} & = 
\sup_{0 \le r \le k} 2^{r-k} \sup_{2^{r} \le |m| < 2^{r+1}} \frac{1}{|m|+1} \left|
\sum_{s=0}^{m} a_s\right| \\
& \le \sup_{0 \le r \le k} 2^{r-k}  \frac{1}{2^{r}} \sum_{s=-(2^{r+1}-1)}^{2^{r+1}-1} |a_s|
\le 2^{-k}   \sum_{s=-(2^{k+1}-1)}^{2^{k+1}-1} |a_s|.
\end{split}
\end{equation*}
Therefore,
\begin{equation*}
\begin{split}
I_1 & \lesssim 
\left(\sum_{k=0}^{\infty} \left(2^{\frac{k}{p'}} 2^{-k}   \sum_{s=-(2^{k+1}-1)}^{2^{k+1}-1} |a_s|\right)^p\right)^{\frac{1}{p}}
= \left(\sum_{k=0}^{\infty} \left(2^{-\frac{k}{p}}    \sum_{s=-(2^{k+1}-1)}^{2^{k+1}-1} |a_s|\right)^p\right)^{\frac{1}{p}}.
\end{split}
\end{equation*}

Let  $-\frac{1}{p'} < \eps < -\frac{1}{p'} + \frac{1}{p}$.
Applying H{\"o}lder's inequality, we get
\begin{equation}\label{2.1}
\begin{split}
I_1 & \lesssim\left(\sum_{k=0}^{\infty} \left(2^{-\frac{k}{p}} \sum_{s=-(2^{k+1}-1)}^{2^{k+1}-1} |a_s|\right)^p\right)^{\frac{1}{p}}\\
& \le
\left(\sum_{k=0}^{+\infty}\left(2^{-\frac{k}{p}} 
\left(\sum_{s=-(2^{k+1}-1)}^{2^{k+1}-1} \left(|a_s|(|s|+1)^{-\eps}\right)^p\right)^{\frac{1}{p}}
\left(\sum_{s=-(2^{k+1}-1)}^{2^{k+1}-1} (|s|+1)^{\eps p'}\right)^{\frac{1}{p'}}
\right)^p\right)^{\frac{1}{p}}\\
& \asymp
\left(\sum_{k=0}^{+\infty}2^{-k} 
\sum_{s=-(2^{k+1}-1)}^{2^{k+1}-1} \left(|a_s|(|s|+1)^{-\eps}\right)^p
\left(2^{k(\eps p' + 1)}\right)^{\frac{p}{p'}}
\right)^{\frac{1}{p}}\\
& =
\left(\sum_{k=0}^{+\infty} 2^{k(\eps p + p-2)}
\sum_{s=-(2^{k+1}-1)}^{2^{k+1}-1} \left(|a_s|(|s|+1)^{-\eps}\right)^p\right)^{\frac{1}{p}}\\
& \asymp \left(\sum_{s=-\infty}^{+\infty} \left(|a_s|(|s|+1)^{-\eps}\right)^p \sum_{k = \log_2 (|s|+1)}^{+\infty} 2^{k(\eps p + p-2)}
\right)^{\frac{1}{p}} \\
& \asymp 
\left(\sum_{s=-\infty}^{+\infty} \left(|a_s|(|s|+1)^{-\eps}\right)^p (|s|+1)^{\eps p + p-2}
\right)^{\frac{1}{p}} = J_p(f). 
\end{split}
\end{equation}
Now we estimate $I_2$.
For any $r \ge k$, we have
\begin{equation*}
\begin{split}
\widehat{a}_{2^r} \le \sum_{t = k}^{\infty} \widehat{a}_{2^t}.
\end{split}
\end{equation*}
Therefore, by using Hardy's inequality \eqref{Hardy1}, we obtain
\begin{equation}\label{2.11}
\begin{split}
I_2 &\le \left(\sum_{k=0}^{\infty} \left(2^{\frac{k}{p'}} \sup_{r \ge k} \widehat{a}_{2^r}\right)^p\right)^{\frac{1}{p}}
\le \left(\sum_{k=0}^{\infty} \left(2^{\frac{k}{p'}}\sum_{t = k}^{\infty} \widehat{a}_{2^t}\right)^p\right)^{\frac{1}{p}}
 \lesssim \left(\sum_{k=0}^{\infty} \left(2^{\frac{k}{p'}}\widehat{a}_{2^{k}}\right)^p\right)^{\frac{1}{p}}\\
& = \left(\sum_{k=0}^{\infty} \left(2^{\frac{k}{p'}}\sup_{2^{k} \le |m| < 2^{k+1}}\frac{1}{|m|+1}\left|\sum_{s=0}^{m} a_s\right|\right)^p\right)^{\frac{1}{p}}
 \le \left(\sum_{k=0}^{\infty} \left(2^{\frac{k}{p'}}\frac{1}{2^{k}}\sum_{s=-2^{k+1}}^{2^{k+1}} |a_s|\right)^p\right)^{\frac{1}{p}}\\
&\lesssim J_p(f).
\end{split}
\end{equation}

Combining inequalities \eqref{2.1} and \eqref{2.11}, we have 
\begin{equation}\label{2.111}
A_p(a) \lesssim J_p(f), \quad 1 < p < \infty.
\end{equation}

For further arguments we consider two cases. 

\noindent {\bf Case 1:} Let $1 < p \le 2$. By Hardy-Littlewood's  inequality \cite{HaLi27} and by Theorem \ref{T4}, we get
\begin{equation}\label{2.2}
\begin{split}
J_p(f) \le \|f\|_{L_p} \lesssim I_p(f). 
\end{split}
\end{equation}
Inequalities \eqref{2.111}
and \eqref{2.2} imply
\begin{equation*}
\begin{split}
A_p(a) \lesssim J_p(f) \le \|f\|_{L_p} 
\lesssim I_p(f).
\end{split}
\end{equation*}
Since $\{a_k\}_{k=-\infty}^{\infty} \in \overline{\textnormal{GM}}$,
\begin{equation*}
A_p(a) \lesssim J_p(f) \le \|f\|_{L_p} 
\lesssim I_p(f) \lesssim  A_p(a) .
\end{equation*}
Therefore,
$$
\|f\|_{L_p} \asymp J_p(f).
$$

\noindent{\bf Case 2:} Let $2 \le p <\infty$. 
For any $m \in \mathbb{Z}$, we consider $t \in \mathbb{N}_0$
such that $[2^{t-1}] \le |m| < 2^{t}$. 
Then, since $\{a_k\}_{k=-\infty}^{\infty} \in \overline{\textnormal{GM}}$,
\begin{equation*}
\begin{split}
|a_m| \le \sum_{|s|\ge |m|}|\Delta a_s|
\le \sum_{k=t}^{\infty}\sum_{[2^{k-1}] \le |s| < 2^{k}}|\Delta a_s|
\le \sum_{k=t}^{\infty} \sup_{l \in \mathbb{N}_0} \min(1,2^{l-k})\widehat{a}_{2^l}.
\end{split}
\end{equation*}
Therefore,
\begin{equation*}
\begin{split}
J_p(f) & = \left(\sum_{m=-\infty}^{\infty} |a_m|^p (|m|+1)^{\frac{p}{p'}-1}\right)^{\frac{1}{p}}
\le 
\left(\sum_{t=0}^{\infty} \sum_{[2^{t-1}] \le |m| < 2^{t}}|a_m|^p(|m|+1)^{\frac{p}{p'}-1}\right)^{\frac{1}{p}}\\
& \lesssim
\left(\sum_{t=0}^{\infty} \left(\sum_{k=t}^{\infty}\sum_{[2^{k-1}] \le |s| < 2^k}|\Delta a_s|\right)^p\sum_{[2^{t-1}] \le |m| < 2^{t}}
(|m|+1)^{\frac{p}{p'}-1}\right)^{\frac{1}{p}}\\
&
\asymp \left(\sum_{t=0}^{\infty} \left(2^{\frac{t}{p'}}\sum_{k=t}^{\infty}\sum_{[2^{k-1}] \le |s| < 2^k}|\Delta a_s|\right)^p\right)^{\frac{1}{p}}\\
&\lesssim
\left(\sum_{t=0}^{\infty} \left(2^{\frac{t}{p'}}\sum_{k=t}^{\infty}\sup_{l \in \mathbb{N}_0} \min(1,2^{l-k})\widehat{a}_{2^l}\right)^p\right)^{\frac{1}{p}}.
\end{split}
\end{equation*}
By using Hardy's inequality \eqref{Hardy1}, we get
\begin{equation*}
\begin{split}
& \left(\sum_{t=0}^{\infty} \left(2^{\frac{t}{p'}}\sum_{k=t}^{\infty}\sup_{l \in \mathbb{N}_0} \min(1,2^{l-k})\widehat{a}_{2^l}\right)^p\right)^{\frac{1}{p}}
\lesssim
\left(\sum_{t=0}^{\infty} \left(2^{\frac{t}{p'}}\sup_{l \in \mathbb{N}_0} \min(1,2^{l-t})\widehat{a}_{2^l}\right)^p\right)^{\frac{1}{p}}\\
&\lesssim
\left(\sum_{t=0}^{\infty} \left(2^{\frac{t}{p'}} \sup_{0 \le l \le t} 2^{l-t}\widehat{a}_{2^l}\right)^p\right)^{\frac{1}{p}}
+ \left(\sum_{t=0}^{\infty} \left(2^{\frac{t}{p'}} \sup_{l \ge t+1} \widehat{a}_{2^l}\right)^p\right)^{\frac{1}{p}}
=: L_1 + L_2.
\end{split}
\end{equation*}
We estimate $L_1$.
\begin{equation*}
\begin{split}
L_1 &= \left(\sum_{t=0}^{\infty} \left(2^{\frac{t}{p'}} \sup_{0 \le l \le t} 2^{l-t}\widehat{a}_{2^l}\right)^p\right)^{\frac{1}{p}}
=
\left(\sum_{t=0}^{\infty} \left(2^{\frac{t}{p'}-t} \sup_{0 \le l \le t} 2^{l}\widehat{a}_{2^l}\right)^p\right)^{\frac{1}{p}}\\
&\le
\left(\sum_{t=0}^{\infty} \left(2^{\frac{t}{p'}-t} \sum_{l=0}^t 2^{l}\widehat{a}_{2^l}\right)^p\right)^{\frac{1}{p}}.
\end{split}
\end{equation*}
By using Hardy's inequality \eqref{Hardy2}, we derive
\begin{equation*}
\begin{split}
L_1 \lesssim \left(\sum_{t=0}^{\infty} \left(2^{\frac{t}{p'}-t} \sum_{l=0}^t 2^{l}\widehat{a}_{2^l}\right)^p\right)^{\frac{1}{p}} \lesssim
\left(\sum_{t=0}^{\infty} \left(2^{\frac{t}{p'}} \widehat{a}_{2^t}\right)^p\right)^{\frac{1}{p}}
\le \|a\|_{n_{p',p}}.
\end{split}
\end{equation*}
Now we estimate $L_2$.
\begin{equation*}
\begin{split}
L_2 &\le  \left(\sum_{t=0}^{\infty} \left(2^{\frac{t}{p'}} \sup_{l \ge t} \widehat{a}_{2^l}\right)^p\right)^{\frac{1}{p}}
= \left(\sum_{t=0}^{\infty} \left(2^{\frac{t}{p'}} \sup_{l \ge t} \sup_{2^l \le |m| < 2^{l+1}}
\frac{1}{|m|+1}\left|\sum_{s=0}^{m} a_s\right|\right)^p\right)^{\frac{1}{p}}\\
&\le
\left(\sum_{t=0}^{\infty} \left(2^{\frac{t}{p'}} \sup_{l \ge t} \sup_{|m| \ge 2^{l}}
\frac{1}{|m|+1}\left|\sum_{s=0}^{m} a_s\right|\right)^p\right)^{\frac{1}{p}}\\
& \le
\left(\sum_{t=0}^{\infty} \left(2^{\frac{t}{p'}} \sup_{|m| \ge 2^{t}}
\frac{1}{|m|+1}\left|\sum_{s=0}^{m} a_s\right|\right)^p\right)^{\frac{1}{p}}
\le \|a\|_{n_{p'p}}.
\end{split}
\end{equation*}
Finally, since $2 \le p < \infty$, by Hardy-Littlewood's inequality \cite{HaLi27},
 we have
\begin{equation*}
\begin{split}
J_p(f) \lesssim \|a\|_{n_{p',p}} \lesssim \|f\|_{L_{p}} \lesssim J_{p}(f).
\end{split}
\end{equation*}
Therefore, $J_p(f) \asymp \|f\|_{L_p}$.
\end{proof}

\section{Comparison of classes of general monotone sequences}\label{comparison}

In this Section, at first, 
we compare classes $\textnormal{GM}_{\mathbb{R}}$ 
and $\overline{\textnormal{GM}}$, where
$$
\textnormal{GM}_{\mathbb{R}} = 
\left\{\{a_k\}_{k=1}^{\infty} \in \textnormal{GM}: \ \ a_k \in \mathbb{R}, \ \ k \ge 1\right\}.
$$ 
For this reason, we need a technique
considered in \cite{DyTi2} by M. Dyachenko and S. Tikhonov.

Without loss of generality, we may assume in the definition of the class $\textnormal{GM}$
 that $\lambda = 2^{\nu}$, where $\nu$ is a natural number.

Let $\{a_k\}_{k=1}^{\infty}\in \textnormal{GM}_{\mathbb{R}}$. 
Denote for any $n \ge 2\nu$
$$
A_n := \max_{2^n \le k \le 2^{n+1}} |a_k|, \quad
B_n := \max_{2^{n-2\nu} \le k \le 2^{n + 2\nu}}|a_k|.
$$

\begin{definition}
Let $\{a_k\}_{k=1}^{\infty} \in \textnormal{GM}_{\mathbb{R}}$. 
We say that a integer number $n \ge 2\nu$ is \textnormal{good},
if 
$B_n \le 2^{2\nu}A_n$.
The rest of integer numbers $n \ge 0$ consists of \textnormal{bad} numbers.
 \end{definition}

Denote
$$
M_n := \left\{k \in [2^{n-\nu}, 2^{n+\nu}] : |a_k| > \frac{A_n}{8C2^{2\nu}} \right\},
$$
and
$$
M_n^+ := \{k \in M_n: a_k > 0\}, \quad M_n^- := M_n \setminus M_n^{+}.
$$
\begin{lemma}[\text{\cite[Lemma 2.2]{DyTi2}}]\label{help}
Let a vanishing sequence $\{a_k\}_{k=1}^{\infty} \in \textnormal{GM}_{\mathbb{R}}$.
Denote $N_0: = [\log_2 (C^32^{10\nu + 8})] +1$.
Then, for any good  $n \ge N_0$, there exists an interval
$[l_n, m_n]\subseteq [2^{n-\nu}, 2^{n+\nu}]$ such that at least
one of the following condition holds:
\begin{enumerate}
\item[\textnormal{(i)}] for any $k \in [l_n, m_n]$, we have $a_k \ge 0$ and
$$
|M_n^+ \cap [l_n, m_n]| \ge \frac{2^n}{C^32^{15\nu+8}};
$$
\item[\textnormal{(ii)}] for any $k \in [l_n, m_n]$, we have $a_k \le 0$ and
$$
|M_n^- \cap [l_n, m_n]| \ge \frac{2^n}{C^32^{15\nu+8}}.
$$
\end{enumerate}
\end{lemma}
\begin{lemma}[\text{\cite[Lemma 3.6]{DyMuTi}}]\label{gb}
Let a vanishing sequence $\{a_k\}_{k=1}^{\infty} \in \textnormal{GM}_{\mathbb{R}}$.
Then for any bad number $n \ge 2\nu$ there exists a set of  integer numbers
\begin{equation*}
n = \xi_0 > \xi_1 > \xi_2 > \ldots > \xi_{s-1} > \xi_{s} =: \xi_{n,s}
\end{equation*}
or
\begin{equation*}
n =\xi_0 < \xi_1 < \xi_2 < \ldots < \xi_{s-1} < \xi_{s}=: \xi_{n,s}
\end{equation*}
such that $\xi_1, \ldots, \xi_{s-1}$ are bad, $\xi_{s}$ is good,  and
$$
A_n < 2^{-2\nu}A_{\xi_1} < 2^{-4\nu}A_{\xi_2}< \ldots < 2^{-2s\nu}A_{\xi_{s}},
$$
$$
|\xi_i - \xi_{i-1}| \le 2\nu  \quad i = 1,\ldots, s.
$$
\end{lemma}
\begin{remark}
These properties of sequences from $\textnormal{GM}_{\mathbb{R}}$ 
were stated in \cite{DyTi2}. Due to these properties, 
some classical results of Fourier analysis and 
approximation theory have been extended to the 
$\textnormal{GM}_{\mathbb{R}}$ class, see \cite{BeDyTi, DyMuTi, DyMuTi2, DyTi2},
and see also \cite{De, De19, DyTi22, Og}.
\end{remark}
\begin{lemma}
Let $\{a_k\}_{k=-\infty}^{\infty}$ be a sequence of real numbers such that
$a_k = 0$, for any $k \le 0$, and $\{a_k\}_{k=1}^{\infty} \in \textnormal{GM}_{\mathbb{R}}$, 
$\lim\limits_{k \to +\infty} a_k = 0$.  
Then, for any $n \ge 0$,
\begin{equation}\label{gmgm-1}
\sum_{k=2^{n}}^{2^{n+1}} \frac{|a_k|}{k} \lesssim 
\sup_{k \in \mathbb{N}_0} \min(1, 2^{k-n}) \widehat{a}_{2^k}.
\end{equation}
\end{lemma}
\begin{proof}
Let $N_0 = N_0(C, \nu)$ be a constant defined in Lemma \ref{help}. 
We prove \eqref{gmgm-1} by considering five cases.

\noindent{\bf Case 1.} Let $1 \le n \le N_0$. Then, for any $2^n \le k \le 2^{n+1}$, we have
\begin{equation*}
\begin{split}
|a_k| & = \left|\sum_{i = 0}^{k} a_i - \sum_{i = 0}^{k-1} a_i \right|
\le  
 (k+1) \cdot \frac{1}{k+1}\left|\sum_{i = 0}^{k} a_i \right| + 
k\cdot \frac{1}{k}\left|\sum_{i = 0}^{k-1} a_i \right| \\
& \le (2^{N_0 + 1}+1) (2\widehat{a}_{2^{n}} + \widehat{a}_{2^{n+1}})
\lesssim \sup_{t \ge n} \widehat{a}_{2^{t}}
\le \sup_{t \in \mathbb{N}_0} \min(1, 2^{t-n})\widehat{a}_{2^{t}}.
\end{split}
\end{equation*}
Therefore,
\begin{equation*}
\begin{split}
\sum_{k = 2^n}^{2^{n+1}} \frac{|a_k|}{k} \lesssim 
\sup_{t \in \mathbb{N}_0} \min(1, 2^{t-n})\widehat{a}_{2^{t}}\sum_{k = 2^n}^{2^{n+1}} \frac{1}{k}
\lesssim \sup_{t \in \mathbb{N}_0} \min(1, 2^{t-n})\widehat{a}_{2^{t}}.
\end{split}
\end{equation*}

\noindent {\bf Case 2.} Let $n > N_0$ be a good number.  
By Lemma \ref{help} there exist integer numbers $l_{n}, m_{n}$
such that $[l_{n}, m_{n}] \subset [2^{n - \nu}, 2^{n + \nu}]$, $m_{n} - l_{n} \asymp 2^{n}$
and
$$
\frac{1}{2^{n}}\left|\sum_{k = l_{n}}^{m_{n}} a_k\right| 
\gtrsim A_{n} \ge \frac{1}{2^{n}} \sum_{k=2^n}^{2^{n+1}} |a_k|
\gtrsim \sum_{k=2^n}^{2^{n+1}} \frac{|a_k|}{k}. 
$$

Consider case, when  $\left|\sum\limits_{k=0}^{l_{n}-1} a_k\right| \ge \frac{1}{2}\left|\sum\limits_{k=l_{n}}^{m_{n}} a_k\right|$, 
then 
$$
\frac{1}{l_{n}}\left|\sum_{k=0}^{l_{n}-1} a_k\right|
\ge \frac{1}{2l_n}\left|\sum_{k=l_{n}}^{m_{n}} a_k\right|
\gtrsim \frac{1}{2^{n}}\left|\sum_{k=l_{n}}^{m_{n}} a_k\right|
\gtrsim  \sum_{k=2^{n}}^{2^{n+1}} \frac{|a_k|}{k}.
$$

Define $\eta$: $2^{\eta} \le l_{n} < 2^{\eta+1}$. 
Then
$$
\sum_{k=2^{n}}^{2^{n+1}} \frac{|a_k|}{k} \lesssim 
\frac{1}{l_{n}}\left|\sum_{k=0}^{l_{n}-1} a_k\right|
\le  \widehat{a}_{2^{\eta}}.
$$
Note that, since $2^{n - \nu} \le l_{n} \le 2^{n + \nu}$, then
$n - \nu \le \eta \le n + \nu - 1$.

Now consider case, when $\left|\sum\limits_{k=0}^{l_{n}-1} a_k\right| < \frac{1}{2}\left|\sum\limits_{k=l_{n}}^{m_{n}} a_k\right|$, 
then 
$$
\left|\sum_{k=0}^{m_{n}}a_k\right| \ge
\left|\sum_{k=l_{n}}^{m_{n}}a_k\right| -
\left|\sum_{k=0}^{l_{n}-1}a_k\right| 
\ge \frac{1}{2}\left|\sum_{k=l_{n}}^{m_{n}}a_k\right|.
$$
Therefore,
$$
\frac{1}{m_{n}+1}\left|\sum_{k=0}^{m_{n}}a_k\right|
\gtrsim  \sum_{k=2^{n}}^{2^{n+1}} \frac{|a_k|}{k}.
$$
Hence,
$$
\sum_{k=2^{n}}^{2^{n+1}} \frac{|a_k|}{k} \lesssim 
 \widehat{a}_{2^{\eta}},
$$
where $\eta$ is defined by relation: $2^{\eta} \le m_{n} + 1 < 2^{\eta+1}$.
As in previous case such $\eta$ satisfies condition: $n - \nu \le \eta \le n + \nu - 1$.

In both cases, we obtain
$$
\sum_{k=2^{n}}^{2^{n+1}} \frac{|a_k|}{k} \lesssim 
 \widehat{a}_{2^{\eta}},
$$
where $n - \nu \le \eta \le n + \nu - 1$.
Therefore,
$$
\sum_{k=2^{n}}^{2^{n+1}} \frac{|a_k|}{k} \lesssim 
\widehat{a}_{2^{\eta}} 
\le  \sup_{k \ge n - \nu}\widehat{a}_{2^{k}}
\lesssim \sup_{k \in \mathbb{N}_0} \min(1, 2^{k-n}) \widehat{a}_{2^k}.
$$
 
Now let $n > N_0$ be a bad number. 
Then,  there exists a good number $\xi_s = \xi_{n,s}$ satisfying conditions
from Lemma \ref{gb}.  In particular, $A_{\xi_s} \ge 2^{2s\nu}A_n$.

\noindent {\bf Case 3.} Let $\xi_s > n$. Then by Lemma \ref{help} there exist integer numbers $l_{\xi_s}, m_{\xi_s}$
such that $[l_{\xi_s}, m_{\xi_s}] \subset [2^{\xi_s - \nu}, 2^{\xi_s + \nu}]$, $m_{\xi_s} - l_{\xi_s} \asymp 2^{\xi_s}$
and
$$
\frac{1}{2^{\xi_s}}\left|\sum_{k = l_{\xi_s}}^{m_{\xi_s}} a_k\right| 
\gtrsim A_{\xi_s}. 
$$
From Lemma \ref{gb} we have
$$
\frac{1}{2^{\xi_s}}\left|\sum_{k = l_{\xi_s}}^{m_{\xi_s}} a_k\right| 
\gtrsim A_{\xi_s} \ge 2^{2s\nu} A_n \gtrsim
2^{2s\nu} \sum_{k=2^n}^{2^{n+1}}\frac{|a_k|}{k}.
$$

Repeating the same arguments as in Case 2,  we derive 
$$
\sum_{k=2^{n}}^{2^{n+1}} \frac{|a_k|}{k} \lesssim
\frac{1}{2\nu s} \cdot \frac{1}{2^{\xi_s}}\left|\sum_{k = l_{\xi_s}}^{m_{\xi_s}} a_k\right|
\lesssim 
\frac{1}{2^{2\nu s}} \widehat{a}_{2^{\eta}},
$$
where $\xi_s - \nu \le \eta \le \xi_s + \nu - 1$. 
Since $\xi_s > n$,  we get $\eta > n - \nu$, and
$$
\sum_{k=2^{n}}^{2^{n+1}} \frac{|a_k|}{k} \lesssim 
\frac{1}{2^{2\nu s}} \widehat{a}_{2^{\eta}} 
\le  \sup_{t > n - \nu}\widehat{a}_{2^{t}}
\lesssim \sup_{t \in \mathbb{N}_0} \min(1, 2^{t-n}) \widehat{a}_{2^t}.
$$

\noindent {\bf Case 4.} Let $N_0 \le \xi_s < n$, then by Lemma \ref{gb}, $\xi_s < n \le \xi_s +2\nu s$.
The same arguments as in the Case 3 imply that there exists integer number $\eta$
such that $\eta - \nu + 1 < n \le \eta + \nu + 2\nu s$ and
\begin{equation*}
\begin{split}
 \sum_{k=2^{n}}^{2^{n+1}} \frac{|a_k|}{k} &\lesssim 
\frac{1}{2^{2\nu s}} \widehat{a}_{2^{\eta}} = 
\frac{2^n 2^{\eta}}{2^{2\nu s}2^n2^{\eta}} \widehat{a}_{2^{\eta}}
= 
\frac{2^n}{2^{2\nu s}2^{\eta}} \cdot \frac{2^{\eta}}{2^{n}} \widehat{a}_{2^{\eta}}\\
& \lesssim \frac{2^{\eta}}{2^n} \widehat{a}_{2^{\eta}}
\le \sup_{t < n + \nu - 1} \frac{2^t}{2^n} \widehat{a}_{2^t}
\lesssim \sup_{t \in \mathbb{N}_0} \min(1,2^{t-n})\widehat{a}_{2^t}.
\end{split}
\end{equation*}

\noindent{\bf Case 5.} Let $\xi_s < N_0 < n$.  Then by Lemma \ref{gb}, $\xi_s < n \le \xi_s +2\nu s$, and
$$
 \sum_{k=2^{n}}^{2^{n+1}} \frac{|a_k|}{k} \lesssim A_n 
\le \frac{1}{2^{2\nu s}}A_{\xi_s}.
$$
The same arguments as in Case 1 imply 
$$
A_{\xi_s} \lesssim (2\widehat{a}_{2^{\xi_s}} + \widehat{a}_{2^{\xi_s + 1}}).
$$
Therefore, 
\begin{equation*}
\begin{split} 
 \sum_{k=2^{n}}^{2^{n+1}} \frac{|a_k|}{k} & \lesssim \frac{1}{2^{2\nu s}}A_{\xi_s} 
\lesssim \frac{1}{2^{2\nu s}} (2\widehat{a}_{2^{\xi_s}} + \widehat{a}_{2^{\xi_s + 1}}) 
 = \frac{2\cdot 2^n 2^{\xi_s}}{2^{2\nu s}\cdot 2^n 2^{\xi_s}} \widehat{a}_{2^{\xi_s}}
+ \frac{2^n 2^{\xi_s + 1}}{2^{2\nu s}\cdot 2^n 2^{\xi_s + 1}} \widehat{a}_{2^{\xi_s+1}}\\
& = \frac{2^{n+1}}{2^{2\nu s}2^{\xi_s}} \cdot \frac{2^{\xi_s}}{2^n} \widehat{a}_{2^{\xi_s}}
+ \frac{2^n}{2^{2\nu s}2^{\xi_s + 1}} \cdot \frac{2^{\xi_s + 1}}{2^n} \widehat{a}_{2^{\xi_s+1}}
\lesssim \sup_{t \le n} 2^{t-n}\widehat{a}_{2^t}
\le \sup_{t \in \mathbb{N}_0} \min(1, 2^{t-n})\widehat{a}_{2^t}.
\end{split}
\end{equation*}

Thus, relation \eqref{gmgm-1} was proved for all good and bad numbers $n$.

\end{proof}

\begin{lemma}
Let $\{a_k\}_{k=-\infty}^{\infty}$ be a sequence of real numbers such that
$a_k = 0$, for any $k \le 0$, and 
$\{a_k\}_{k=1}^{\infty} \in \textnormal{GM}_{\mathbb{R}}$,
$\lim\limits_{k \to +\infty} a_k = 0$.  
Then $\{a_k\}_{k=-\infty}^{\infty} \in \overline{\textnormal{GM}}$. 
In this sense,
we interpret the following inclusion:
$$
\textnormal{GM}_{\mathbb{R}} \subset
\overline{\textnormal{GM}}.
$$ 
\end{lemma}
\begin{proof}
Let $n \in \mathbb{N}$.  Then, since 
$\{a_k\}_{k=1}^{\infty} \in \textnormal{GM}$, we have 
$$
\sum_{[2^{n-1}] \le |m| < 2^n} |\Delta a_m|  = 
\sum_{2^{n-1} \le m < 2^n} |\Delta a_m| 
\lesssim \sum_{k = 2^{n-\nu -1}}^{2^{n+\nu - 1}} \frac{|a_k|}{k}
\le  \sum_{s = n - \nu - 1}^{n+\nu - 2}\sum_{k = 2^s}^{2^{s+1}} \frac{|a_k|}{k}.
$$
By using inequality \eqref{gmgm-1} 
\begin{equation*}
\begin{split}
\sum_{[2^{n-1}] \le |m| < 2^n} |\Delta a_m| 
&\lesssim \sum_{s = n - \nu - 1}^{n+\nu - 2}\sum_{k = 2^s}^{2^{s+1}} \frac{|a_k|}{k}
\lesssim \sum_{s = n - \nu - 1}^{n+\nu - 2} \sup_{k \in \mathbb{N}_0} \min(1,2^{k-s})\widehat{a}_{2^k} \\
&\le 2^{\nu+1}\nu \sup_{k \in \mathbb{N}_0} \min(1,2^{k-n})\widehat{a}_{2^k}.
\end{split}
\end{equation*}
Therefore, for any $n \ge 1$, we obtain
\begin{equation}\label{gmgm-2}
\sum_{[2^{n-1}] \le |m| < 2^n} |\Delta a_m| 
\lesssim \sup_{k \in \mathbb{N}_0} \min(1,2^{k-n})\widehat{a}_{2^k}.
\end{equation}
It easy to see, that inequality \eqref{gmgm-2} holds for $n = 0$.
Hence, $\{a_n\}_{n= - \infty}^{\infty} \in \overline{\textnormal{GM}}$.
\end{proof}

According to the notations in \cite{GrSaSa},  we consider sector
of the complex plane:
$$
S_{\alpha, \beta} = \{z \in \mathbb{C} : |\arg z -\alpha| \le \beta \}, 
\quad 0 \le \alpha < 2\pi, \quad 0 \le \beta < \frac{\pi}{2}.
$$
Now we compare the classes $\textnormal{GM}_{\alpha, \beta}$
and $\overline{\textnormal{GM}}$, where
$$
\textnormal{GM}_{\alpha, \beta} = 
\left\{\{a_k\}_{k=1}^{\infty} \in  \textnormal{GM}:  \ \ a_k \in S_{\alpha, \beta}, \ \ k \ge 1\right\}.
$$
 
\begin{lemma}
Let $0 \le \alpha < 2\pi$,
$0 \le \beta < \frac{\pi}{2}$, 
and  let $\{a_k\}_{k=-\infty}^{\infty}$ 
be a sequence of complex numbers such that
$a_k = 0$, for any $k \le 0$, and 
$\{a_k\}_{k=1}^{\infty} \in \textnormal{GM}_{\alpha, \beta}$.  
Then $\{a_k\}_{k=-\infty}^{\infty} \in \overline{\textnormal{GM}}$. In this sense,
we interpret the following inclusion:
$$
\textnormal{GM}_{\alpha, \beta} \subset
\overline{\textnormal{GM}}.
$$
\end{lemma}
\begin{proof}
It is sufficient to prove that, for any $n \ge 0$, 
$$
\sum_{k=2^{n}}^{2^{n+1}} \frac{|a_k|}{k} \lesssim 
\sup_{k \in \mathbb{N}_0} \min(1, 2^{k-n}) \widehat{a}_{2^k}
$$

By  \cite[Lemma 4.2]{GrSaSa}, we have
$$
\sum_{k = 2^{n}}^{2^{n+1}} |a_k| \le \frac{1}{|\cos \beta|} \left|\sum_{k = 2^{n}}^{2^{n+1}} a_k \right|.
$$
Therefore, 
\begin{equation*}
\begin{split}
\sum_{k = 2^{n}}^{2^{n+1}}\frac{|a_k|}{k}
\lesssim \frac{1}{2^n} \sum_{k = 2^{n}}^{2^{n+1}} |a_k| 
\le \frac{1}{2^n |\cos \beta|} \left|\sum_{k = 2^{n}}^{2^{n+1}} a_k \right|.
\end{split}
\end{equation*}
Consider case, when  $\left|\sum\limits_{k=0}^{2^{n}-1} a_k\right| \ge \frac{1}{2}\left|\sum\limits_{k=2^{n}}^{2^{n+1}} a_k\right|$, 
then 
$$
\frac{1}{2^{n}}\left|\sum\limits_{k=0}^{2^{n}-1} a_k\right|
\ge \frac{1}{2^{n+1}}\left|\sum\limits_{k=2^{n}}^{2^{n+1}} a_k\right|
\gtrsim  \sum_{k=2^{n}}^{2^{n+1}} \frac{|a_k|}{k}.
$$
Therefore,
$$
\widehat{a}_{2^{n}} \gtrsim \sum_{k=2^{n}}^{2^{n+1}} \frac{|a_k|}{k}.
$$
Hence, 
\begin{equation*}
\begin{split}
\sum_{k=2^{n}}^{2^{n+1}} \frac{|a_k|}{k} \lesssim
\widehat{a}_{2^{n}}  \le 
\sup_{t \in \mathbb{N}_0} \min(1, 2^{t-n}) \widehat{a}_{2^t}. 
\end{split}
\end{equation*}

Now consider case, when  $\left|\sum\limits_{k=0}^{2^{n}-1} a_k\right| < \frac{1}{2}\left|\sum\limits_{k=2^{n}}^{2^{n+1}} a_k\right|$, 
then 
$$
\left|\sum\limits_{k=0}^{2^{n+1}} a_k\right|
\ge \left|\sum\limits_{k=2^n}^{2^{n+1}} a_k\right|
-\left|\sum\limits_{k=0}^{2^{n}-1} a_k\right| 
> \frac{1}{2} \left|\sum\limits_{k=2^n}^{2^{n+1}} a_k\right|.
$$
Hence,
$$
\frac{1}{2^{n+1}+1}\left|\sum\limits_{k=0}^{2^{n+1}} a_k\right|
\gtrsim \frac{1}{2^{n}}\left|\sum\limits_{k=2^{n}}^{2^{n+1}} a_k\right|
\gtrsim  \sum_{k=2^{n}}^{2^{n+1}} \frac{|a_k|}{k}.
$$
Thus,
$$
\sum_{k=2^{n}}^{2^{n+1}} \frac{|a_k|}{k} 
\lesssim 
\widehat{a}_{2^{n+1}} \le 
 \sup_{t \in \mathbb{N}_0} \min(1, 2^{t-n}) \widehat{a}_{2^t}.
$$
\end{proof}
\begin{prop}\label{compare2}
There exists a sequence $\{a_k\}_{k=-\infty}^{\infty} \in \overline{\textnormal{GM}}$ 
such that $a_k = 0$, for any $k \le 0$ and 
$\{a_k\}_{k=1}^{\infty} \notin \textnormal{GM}$. In this sense we interpret the following
relation:
$$
\overline{\textnormal{GM}} \setminus 
\textnormal{GM} \neq \varnothing.
$$ 
\end{prop}
\begin{proof}
Consider a sequence $\{a_k\}_{k=-\infty}^{\infty}$ defined as follows
$$
a_k = \\
\begin{cases}
0, & k \le 15, \\
\frac{1}{2^n}, & 2^n \le k < 2^{n+1}, \ \  n \ge 4, \ \ k \ \text{is odd};\\
0, & 2^n \le k < 2^n + [\sqrt{n}], \ \ n \ge 4, \ \ k \ \text{is even};\\
\frac{1}{2^n}, & 2^n + [\sqrt{n}] \le k < 2^{n+1}, \ \ n \ge 4, \ \ k \ \text{is even};   
\end{cases}
$$
For any $n \ge 4$, we have
\begin{equation*}
\begin{split}
\sum_{k=2^n}^{2^{n+1}-1}|\Delta a_k| \asymp
\sum_{k=2^n}^{2^n + [\sqrt{n}]}\frac{1}{2^n} \asymp  \frac{\sqrt{n}}{2^n}. 
\end{split}
\end{equation*}
On the other hand, for given $n \ge 4$, $\lambda > 1$,
we choose $m, l \in \mathbb{N}$ such that $2^m \le \frac{2^n}{\lambda} < 2^{m+1}$,
$2^{l-1} < \lambda 2^n \le 2^l$.
Hence,
  
\begin{equation*}
\begin{split}
& \frac{1}{2^n} \sum_{k = \frac{2^n}{\lambda}}^{\lambda 2^n}|a_k|
\le \frac{1}{2^n} \sum_{k = 2^m}^{2^l}|a_k| 
\le \frac{1}{2^n} \left(\sum_{s = m}^{l-1}\sum_{k = 2^s}^{2^{s+1}-1}|a_k| + |a_{2^l}|\right)
 \le \frac{1}{2^n} (l-m+1) 
\le  \frac{3+ 2\log_2 \lambda}{2^n}.
\end{split}
\end{equation*}
Therefore, $\{a_k\}_{k=1}^{\infty} \notin \textnormal{GM}$.

On the other hand, for sufficient large $n$, we have
\begin{equation*}
\begin{split}
\widehat{a}_{2^n} & \ge \frac{1}{2^n} \left|\sum_{k  = 0}^{2^n-1} a_k\right| \ge  
\frac{1}{2^n} \left(\sum_{s = 4}^{n-1} \sum_{k  = 2^s}^{2^{s+1}-1} a_k\right)
\ge \frac{1}{2^n} \left(\sum_{s = 4}^{n-1} \frac{2^s - [\sqrt{s}]}{2^s}\right)\\
& \ge \frac{1}{2^n} \left(n  - 4 - \sum_{s = 0}^{n-1}\frac{[\sqrt{s}]}{2^s}\right)
\gtrsim \frac{n}{2^n},\\
\end{split}
\end{equation*}
since series $\sum\limits_{n=0}^{\infty}\frac{[\sqrt{n}]}{2^n}$ is converging.

Hence, 
\begin{equation*}
\begin{split}
& \sup_{k \in \mathbb{N}_0} \min(1, 2^{k-n}) \widehat{a}_{2^k} 
\ge \widehat{a}_{2^n} \gtrsim \frac{n}{2^n}.
\end{split}
\end{equation*}
Therefore, $\{a_k\}_{k=-\infty}^{\infty} \in \overline{\textnormal{GM}}$.
\end{proof}

\section{Compensatory effect of  sequences in $\overline{\textnormal{GM}}$}\label{compensatory-effect}
The following example demonstrates that previously established generalizations 
of the Hardy-Littlewood theorem do not apply to certain series of the form  
$
\sum\limits_{k=1}^{\infty} (a_k + i b_k) e^{ikx}
$  
even while Theorems \ref{T5} and \ref{T6} remain applicable.

\begin{prop}\label{compare1}
There exist $\{a_k\}_{k=1}^{\infty} \notin \textnormal{GM}_{\mathbb{R}}$, 
$\{b_k\}_{k=1}^{\infty} \in \textnormal{GM}_{\mathbb{R}}$,  such that
$\{c_k\}_{k=-\infty}^{\infty} \in \overline{\textnormal{GM}}$, where 
$$
c_k = 
\begin{cases}
a_k + ib_k, & k > 0,\\
0, & k \le 0.
\end{cases} 
$$
\end{prop}
\begin{proof}
Consider sequence $\{a_k\}_{k=1}^{\infty}$ defined as follows 
$$
a_k = (-1)^k \frac{1}{2^{\frac{7}{4}n}},\ \ 
\text{for} \ \ 2^n \le k < 2^{n+1}, \ n \in \mathbb{N}_0.
$$

First, we consider
\begin{equation}\label{eq1}
\begin{split}
\sum_{k = 2^n}^{2^{n+1}}|\Delta a_k| & >
\sum_{k = 2^n}^{2^{n+1}-1}\frac{2}{2^{\frac{7}{4}n}}  = 
\frac{2}{2^{\frac{3}{4}n}}.  
\end{split}
\end{equation}
On the other hand, for given $n \in \mathbb{N}$, $\lambda > 1$, define $m, l \in \mathbb{N}$
such that $2^m \le \frac{2^n}{\lambda} < 2^{m+1}$, $2^{l-1} < \lambda 2^n \le 2^l$.
Hence,
\begin{equation}\label{eq2}
\begin{split}
 \frac{1}{2^n}\sum_{k = \frac{2^n}{\lambda}}^{\lambda 2^n}|a_k|
&\le \frac{1}{2^n}\sum_{k = 2^m}^{2^l}|a_k| 
= \frac{1}{2^n}\left(\sum_{s=m}^{l-1}\sum_{k = 2^s}^{2^{s+1}-1}|a_k|+|a_{2^{l}}|\right) \\
& = \frac{1}{2^n}\left(\sum_{s=m}^{l-1} \frac{1}{2^{\frac{3}{4}s}}+ \frac{1}{2^{\frac{7}{4}l}}\right) < \frac{4}{2^n}. 
\end{split}
\end{equation}
From \eqref{eq1} and \eqref{eq2} follows that  
there is no $C > 0$, $\lambda > 1$ such that inequality
$$
\sum_{k = 2^n}^{2^{n+1}} |\Delta a_k| \le 
\frac{C}{2^n}\sum_{k = \frac{2^n}{\lambda}}^{\lambda 2^n} |a_k|
$$
holds for any $n \in \mathbb{N}$.

Now we consider sequence $\{b_k\}_{k=1}^{\infty}$ defined as follows
$$
b_k = \left(\frac{2}{3}\right)^n \ \ 
\text{for} \ \ 2^n \le k < 2^{n+1}, \ \ n \in \mathbb{N}_0.
$$
The sequence $\{b_k\}_{k=1}^{\infty}$ is nonincreasing sequence, hence,
$\{b_k\}_{k=1}^{\infty} \in \textnormal{GM}$.

Finally, we consider sequence $\{c_k\}_{k=-\infty}^{\infty}$.
Let $n \in \mathbb{N}$. Then
\begin{equation*}
\begin{split}
 \sum_{k=2^n}^{2^{n+1}-1}|\Delta c_k| & = 
   \sum_{k=2^n}^{2^{n+1}-1}\sqrt{(\Delta a_k)^2 + (\Delta b_k)^2} \\
& =  \sum_{k=2^n}^{2^{n+1}-2}\sqrt{\frac{4}{2^{\frac{7}{2}n}}} + 
\sqrt{\left(\frac{1}{2^{\frac{7}{4}n}} + \frac{1}{2^{\frac{7}{4}(n+1)}}\right)^2 + 
\left(\left(\frac{2}{3}\right)^n - \left(\frac{2}{3}\right)^{n+1}\right)^2}\\
& \le \frac{2}{2^{\frac{3}{4}n}} + \sqrt{\frac{4}{2^{\frac{7}{2}n}} + \left(\frac{4}{9}\right)^n}
\le \frac{2}{2^{\frac{3}{4}n}} + 3\left(\frac{2}{3}\right)^n
< 5\left(\frac{2}{3}\right)^n.
\end{split}
\end{equation*}
On the other hand,
\begin{equation*}
\begin{split}
\widetilde{c}_{2^n} & \ge \frac{1}{2^n}\left|\sum_{j = 1}^{2^n-1} c_k\right|
= \frac{1}{2^n} \sqrt{\left(\sum_{k=0}^{n-1} \sum_{j = 2^k}^{2^{k+1}-1} a_j\right)^2 + 
\left(\sum_{k=0}^{n-1} \sum_{j = 2^k}^{2^{k+1}-1} b_j\right)^2}
\\
& = \frac{1}{2^n} \sqrt{\left(\sum_{k=0}^{n-1} \sum_{j = 2^k}^{2^{k+1}-1} (-1)^j \frac{1}{2^{\frac{7}{4}k}}\right)^2 + 
\left(\sum_{k=0}^{n-1} \sum_{j = 2^k}^{2^{k+1}-1} \left(\frac{2}{3}\right)^k\right)^2} \\
& = \frac{1}{2^n} \sqrt{1 + 
\left(\sum_{k=0}^{n-1}  \left(\frac{4}{3}\right)^k\right)^2}
\ge \frac{1}{2^n} \sqrt{1 + \left(\frac{4}{3}\right)^{2n-2}}
\ge \frac{3}{4}\left(\frac{2}{3}\right)^n.
\end{split}
\end{equation*}
Therefore, for any $n \in \mathbb{N}$,
$$
\sum_{k=2^n}^{2^{n+1}-1}|\Delta c_k| 
\le \frac{20}{3} \widehat{c}_{2^n} 
\lesssim \sup_{k \in \mathbb{N}_0} \min (1, 2^{k-n}) \widehat{c}_{2^k}.
$$
It is easy to check that this inequality holds for $n = 0$.
Hence, $\{c_k\}_{k=-\infty}^{\infty} \in \overline{\textnormal{GM}}$.
\end{proof}

The following example demonstrates that previously established generalizations 
of the Hardy-Littlewood theorem do not apply to certain series of the form  
$
\sum\limits_{k=-\infty}^{\infty} c_k e^{ikx}
$  
even while Theorems \ref{T5} and \ref{T6} remain applicable.
\begin{prop}
There exists $\{c_k\}_{k = -\infty}^{\infty}$ such that $\{c_k\}_{k=-\infty}^{0} \notin \overline{\textnormal{GM}}$, 
$\{c_k\}_{k = 1}^{\infty} \in \overline{\textnormal{GM}}$, but 
$\{c_k\}_{k=-\infty}^{\infty} \in \overline{\textnormal{GM}}$.
\end{prop}
\begin{proof}
For integer $k \ge 0$, let's take
$$
c_k 
= \begin{cases}
\left(\frac{2}{3}\right)^n, & \text{for} \ \ 2^{n} \le k < 2^{n+1}, n \ge 0 \\
\left(\frac{2}{3}\right)^n, & \text{for} \ \ k = -2^n, n \ge 0 \\
0, & \text{for} \ \ -2^{n+1} < k < -2^{n}, n \ge 0 \\
0, & \text{for} \ \ k = 0\\
\end{cases}
$$
It is easy to see that, $\{c_k\}_{k=0}^{\infty}$ is a monotone sequence, 
whereas $\{c_k\}_{k=-\infty}^{0}$ is  lacunary. 

First consider the sequence $\{c_k\}_{k = -\infty}^{0}$. 
For $n \ge 1$, we have
$$
\sum_{i = -2^{n} + 1}^{-2^{n-1}} |\Delta c_i| = 
\sum_{i = -2^{n} + 1}^{-2^{n-1}} |c_i - c_{i-1}| = \left(\frac{2}{3}\right)^n.
$$
Let  $k \ge 1$ and $-2^k < m \le -2^{k-1}$, then
$$
\frac{1}{|m|+1} \left|\sum_{j = m}^{0} c_j\right| \le 
\frac{1}{2^{k-1}} \left(\sum_{j = -2^k}^{0} c_j\right) \le 
\frac{1}{2^{k-1}} \sum_{j = 0}^{k} \left(\frac{2}{3}\right)^j  
\lesssim \frac{1}{2^k}.
$$
Therefore, for any $k \ge 1$,
$$
{\widehat{c}_{2^k}}^- := 
\sup_{-2^k < m \le -2^{k-1}} \frac{1}{|m|+1} \left|\sum_{j = m}^{0} c_j\right|
\lesssim \frac{1}{2^{k}}.
$$
It is easy to see, that 
$$
\widehat{c}_{2^0}^- = 
\sup_{-2^0 < m \le -[2^{-1}]} \frac{1}{|m|+1} \left|\sum_{j = m}^{0} c_j\right| 
= c_0 = 0.
$$
Hence, 
\begin{equation*}
\begin{split}
\sup_{k \in \mathbb{N}_0} \min(1, 2^{k-n}) \widehat{c}_{2^k}^- &= 
\max\left\{\max_{0\le k \le n} 2^{k-n}\widehat{c}_{2^k}^-, \sup_{k > n}\widehat{c}_{2^k}^-\right\}\\
& \le \max\left\{\max_{0\le k \le n} 2^{k-n}\frac{1}{2^k}, \sup_{k > n}\frac{1}{2^k}\right\} 
\le \frac{1}{2^{n}}.
\end{split}
\end{equation*}
Therefore, $\{c_k\}_{k = -\infty}^{0} \notin \overline{\textnormal{GM}}$.

Now consider the sequence $\{c_k\}_{k = 0}^{+\infty}$. Let $n \ge 1$, one can see that
$$
\sum_{i = 2^{n-1}}^{2^{n}-1} |\Delta c_i| = 
\sum_{i = 2^{n-1}}^{2^{n} - 1} |c_i - c_{i+1}| \lesssim \left(\frac{2}{3}\right)^n.
$$

On the other hand, we have
\begin{equation*}
\begin{split}
& \sup_{k \in \mathbb{N}_0} \min(1, 2^{k-n}) \widehat{c}_{2^k}^+ \ge 
\widehat{c}_{2^{n+2}}^+ = 
\sup_{2^{n+1} \le m < 2^{n+2}} \frac{1}{m+1} \left|\sum_{j=0}^{m} c_j\right|
\ge \frac{1}{2^{n+2}} \left(\sum_{j=0}^{2^{n+1}} c_j\right)\\
& \ge \frac{1}{2^{n+2}} \left(\sum_{s=0}^{n} \sum_{j = 2^s}^{2^{s+1}-1} \left(\frac{2}{3}\right)^s\right)
= \frac{1}{2^{n+2}} \left(\sum_{s=0}^{n} \left(\frac{4}{3}\right)^s\right)\\
& = \frac{3}{2^{n+2}} \cdot \left(\left(\frac{4}{3}\right)^{n+1}-1\right) \gtrsim
\frac{3}{2^{n+2}} \cdot \left(\frac{4}{3}\right)^{n} = \frac{3}{4}\cdot\left(\frac{2}{3}\right)^n. 
\end{split}
\end{equation*}
Therefore, $\{c_k\}_{k=0}^{\infty} \in \overline{\textnormal{GM}}$.
And, moreover, $\{c_k\}_{k=-\infty}^{\infty} \in \overline{\textnormal{GM}}$.
\end{proof}

\section{Alternating series and idempotent multipliers}\label{alternating-series}
We consider slightly modified notion of idempotent Fourier multiplier. 

\begin{definition}
We will say that a sequence of  numbers $\{\lambda_k\}_{k=-\infty}^{\infty}$  
is  an idempotent multiplier in $L_p([-\pi,\pi])$, $1 < p < \infty$ if:  
\begin{enumerate}  
\item[1.] $\lambda_k \in \{-1,1\}$, for all $k \in \mathbb{Z}$;  
\item[2.] The sequence $\{\lambda_k\}_{k=-\infty}^{\infty}$ is an $L_p$-multiplier, i.e., for any function $f \in L_p([-\pi, \pi])$  
with the Fourier series $\sum\limits_{k=-\infty}^{\infty} c_k e^{ikx}$,  
there exists a function $f_{\lambda} \in L_p([-\pi,\pi])$ with the Fourier series  
$\sum\limits_{k=-\infty}^{\infty} \lambda_k c_k e^{ikx}$ such that  
$\|f_{\lambda}\|_p \lesssim \|f\|_p$.  
\end{enumerate}
\end{definition}

\begin{example}
One can see  that,  
a sequence $\lambda_k = (-1)^k$, $k \in \mathbb{Z}$,
is the idempotent $L_p$-multiplier. Indeed,  for a function 
$f(x) \sim \sum\limits_{k = -\infty}^{\infty}c_k e^{ikx}$, 
a series $\sum\limits_{k = -\infty}^{\infty}(-1)^k c_k e^{ikx} = 
\sum\limits_{k = -\infty}^{\infty}c_k e^{ik(x+\pi)}$ is a Fourier series of the function 
$f_{\lambda}(x) = f(x + \pi)$.
Since, $\|f(\cdot)\|_{L_p} = \|f(\cdot + \pi)\|_{L_p}$, 
we conclude that $\{(-1)^k\}_{k=-\infty}^{\infty}$
is the idempotent $L_p$-multiplier. 
\end{example}

\begin{example}
By Marcinkiewicz's multiplier theorem \cite{Ma},  
the sequence $\{\lambda_k\}_{k=-\infty}^{\infty}$ defined by  
$$
\lambda_k = 
\begin{cases}
(-1)^n, &  \text{for} \ \ 2^{n-1} \leq |k| < 2^n, \  \  n \ge 1\\
1, &  \text{for} \ \ k = 0
\end{cases}
$$
is an idempotent $L_p$-multiplier.
\end{example}

For any idempotent $L_p$-multiplier $\lambda = \{\lambda_k\}_{k=-\infty}^{\infty}$, since 
$\lambda_k^2 = 1$ and $|\lambda_k| = 1$, we establish:
\begin{enumerate}
\item[1.] $J_p(f) = J_p(f_{\lambda})$;
\item[2.]  $\|f\|_{L_p} \asymp \|f_{\lambda}\|_{L_p}$.
\end{enumerate}
These properties imply the following Corollary.

\begin{corollary}\label{corol1}
Let $1 < p < \infty$, $\lambda = \{\lambda_k\}_{k = -\infty}^{\infty}$ be 
an idempotent $L_p$-multiplier, and $f(x) \sim \sum\limits_{k=-\infty}^{\infty} c_k e^{ikx}$
be an integrable on $[-\pi, \pi]$ function. If $\lambda c = 
\{\lambda_k c_k\}_{k = -\infty}^{\infty} \in \overline{\textnormal{GM}}$, then
$$
\|f\|_{L_p} \asymp J_p(f).
$$ 
\end{corollary}

\end{document}